\documentclass[11pt]{article}
\usepackage{amsfonts} 
\usepackage{mathrsfs} 
\usepackage{pstricks-add}
\usepackage{amssymb, amsmath, amsfonts}
\usepackage{makeidx}
\usepackage{color}
\usepackage[T1]{fontenc}
\usepackage{pstricks, pst-node}
\usepackage{graphics,graphicx,graphpap}
\usepackage[ansinew]{inputenc}
\usepackage{fancyhdr}

\if@twoside \oddsidemargin 5pt \evensidemargin 5pt \marginparwidth 20pt
 \else \oddsidemargin 5pt \evensidemargin 5pt \marginparwidth 20pt \fi

\newrgbcolor{qqwuqq}{0. 0.39215686274509803 0.}
\newrgbcolor{dcrutc}{0.8627450980392157 0.0784313725490196 0.23529411764705882}
\newrgbcolor{xfqqff}{0.4980392156862745 0. 1.}

\newcommand{\Aut}{\mathop{\rm Aut}}

\newcommand{\m}{{\cal M}}
\newcommand{\ZZ}{\mathbb{Z}}
\newcommand{\G}{\Gamma}
\newcommand{\fl}{\cal F}
\newcommand{\2}{2_{\{0,1\}}}

\newtheorem{theorem}{Theorem}[section]
\newtheorem{lemma}[theorem]{Lemma}

\newtheorem{proposition}[theorem]{Proposition}
\newtheorem{corollary}[theorem]{Corollary}

\title{Arc Transitive Maps with underlying Rose Window Graphs}
\author{Isabel Hubard, Alejandra Ramos Rivera, Primo\v{z} \v{S}parl}

\begin{document}

\begin{center}
\Large{\textbf{Arc Transitive Maps with underlying Rose Window Graphs}} \\ [+4ex]
Isabel Hubard{\small$^{a,}$\footnotemark $^{,*}$}, Alejandra Ramos Rivera{\small$^{b,}$\footnotemark}, \
Primo\v z \v Sparl{\small$^{b, c, d,}$\footnotemark}  
\\ [+2ex]
{\it \small 
$^a$Instituto de Matem\'aticas, Universidad Nacional Aut\'onoma de M\'exico (UNAM), Mexico City,
Mexico\\
$^b$University of Primorska, IAM, Muzejski trg 2, 6000 Koper, Slovenia\\
$^c$University of Ljubljana, Faculty of Education, Kardeljeva plo\v s\v cad 16, 1000 Ljubljana, Slovenia\\
$^d$IMFM, Jadranska 19, 1000 Ljubljana, Slovenia}
\end{center}

\addtocounter{footnote}{-2}
\footnotetext{Supported in part by ...}
\addtocounter{footnote}{1}\footnotetext{Supported in part by the Slovenian Research Agency (research program P1-0285 and Young Researchers Grant).}
\addtocounter{footnote}{1} \footnotetext{
Supported in part by the Slovenian Research Agency (research program P1-0285 and research projects N1-0038, J1-6720, J1-7051).

Email addresses: 
isahubard@im.unam.mx (Isabel Hubard), 
alejandra.rivera@iam.upr.si (Alejandra Ramos Rivera),
primoz.sparl@pef.uni-lj.si (Primo\v z \v Sparl).

~*Corresponding author  }

\begin{abstract}
Let $\m$ be a map with the underlying graph $\G$. The automorphism group $\Aut(\m)$ induces a natural action on the set of all vertex-edge-face incident triples, called {\em flags} of $\m$. The map $\m$ is said to be a {\em $k$-orbit} map if $\Aut(\m)$ has $k$ orbits on the set of all flags of $\m$. It is known that there are seven different classes of $2$-orbit maps, with only four of them corresponding to arc-transitive maps, that is maps for which $\Aut(\m)$ acts arc-transitively on the underlying graph $\G$. The Petrie dual operator links these four classes in two pairs, one of which corresponds to the chiral maps and their Petrie duals. 

In this paper we focus on the other pair of classes of $2$-orbit arc-transitive maps. We investigate the connection of these maps to consistent cycles of the underlying graph with special emphasis on such maps of smallest possible valence, namely $4$. We then give a complete classification of such maps whose underlying graphs are arc-transitive Rose Window graphs.         
\end{abstract}

\section{Introduction}

\subsection{Arc transitive graphs and consistent cycles}
\label{subsec:AT graphs}

A graph $\G$ without multiple edges is {\em arc-transitive} if its automorphism group $\Aut(\G)$ acts transitively on the set $A(\G)$ of all ordered pairs $(u,v)$ of adjacent vertices of $\Gamma$, called {\em arcs} (and so each edge corresponds to two arcs). When studying arc-transitive graphs investigation of certain cycles, called {\em consistent cycles}, may give an insight into the structure of the graph in question. The notion of consistent cycles was introduced by Conway in 1971 and has recently been studied in various families of graphs and other combinatorial structures (see for instance~\cite{BobMikPot11, Mik13, MikPotWil07} and the references therein). 

Let $\G$ be a graph admitting an arc-transitive group of automorphisms $G \leq \Aut(\G)$. A directed (but not rooted) cycle $\vec{C} = (v_0,v_1,\ldots , v_{r-1})$ of $\Gamma$ is said to be {\em $G$-consistent} if there exists $g \in G$ mapping each $v_i$ to $v_{i+1}$ (where the indices are computed modulo $r$). For us a directed cycle is thus nothing but a connected subgraph of valence $2$ together with one of its two possible orientations. In this case $g$ is said to be a {\em shunt} of $\vec{C}$. Of course, the {\em inverse} $\vec{C}^{-1} = (v_0,v_{r-1},v_{r-2}, \ldots, v_1)$ is $G$-consistent if and only if $\vec{C}$ is $G$-consistent. Thus an (undirected) cycle is said to be $G$-consistent if both of its two corresponding directed cycles are $G$-consistent. 

Suppose $\vec{C}$ is a $G$-consistent directed cycle. It may happen that there is an automorphism in $G$, mapping $\vec{C}$ to $\vec{C}^{-1}$. In such a case we say that the underlying undirected cycle $C$ of $\vec{C}$ is a {\em $G$-symmetric} consistent cycle. Otherwise it is a {\em $G$-chiral} consistent cycle. It is well known and easy to see that $G$ induces a natural action on the set of all $G$-consistent (directed) cycles. The above remarks thus imply that each $G$-orbit of $G$-symmetric consistent cycles corresponds to one $G$-orbit of $G$-consistent directed cycles, while each $G$-orbit of $G$-chiral consistent cycles corresponds to two such orbits. Moreover, the following has been proved in~\cite[Corollary~5.2]{MikPotWil07}.

\begin{proposition}{\rm (\cite{MikPotWil07})}
\label{pro:cons}
Let $\G$ be a graph of valency $k$ admitting an arc-transitive group of automorphisms $G$ and let $s$ and $c$ denote
the numbers of $G$-orbits of $G$-symmetric and $G$-chiral consistent cycles, respectively. Then $s+2c = k-1$. In particular, if $k$ is even then $\G$ contains at least one $G$-orbit of $G$-symmetric consistent cycles.
\end{proposition}

In this paper we will be studying graphs admitting a group of automorphisms acting regularly on their arc-set (we say that such a group is {\em $1$-regular}). For such situations the following observation, which is an immediate corollary of \cite[Corollary~2.3]{Mik13}, is useful.

\begin{lemma}
\label{le:cons}
Let $\G$ be graph admitting a $1$-regular group of automorphisms $G$ and let $e$ be an edge of $\G$. Then each $G$-orbit of $G$-consistent directed cycles of $\G$ contains exactly one $G$-consistent directed cycle containing $e$.
\end{lemma}

In the case that the graph under consideration is tetravalent we can say more.

\begin{lemma}
\label{le:tetra_cons}
Let $\G$ be a tetravalent graph admitting a $1$-regular subgroup $G$ of automorphisms. Then $\G$ has three $G$-orbits of $G$-consistent cycles, all of which are $G$-symmetric, if and only if the vertex stabilizers in $G$ are isomorphic to the Klein $4$-group.
\end{lemma}

\begin{proof}
Let $v \in V(\G)$ be a vertex of $\G$. Since $\G$ is tetravalent and $G$ is $1$-regular the vertex stabilizer $G_v$ is isomorphic either to $\ZZ_4$ or $\ZZ_2 \times \ZZ_2$. 

Suppose first that $G_v = \langle \beta \rangle \cong \ZZ_4$, let $u$ be a neighbor of $v$ and let $u_i = u\beta^i$ for $i \in \ZZ_4$. Since $G$ is $1$-regular there exists an automorphism $\gamma \in G$ mapping the arc $(u_0,v)$ to the arc $(v,u_1)$. Then $\gamma$ is a shunt of a directed $G$-consistent cycle $\vec{C}$ containing $(u_0,v,u_1)$. If the underlying cycle $C$ was $G$-symmetric, there would exist an automorphism $\delta \in G_v$ interchanging $u_0$ and $u_1$, which is impossible as $G_v = \langle \beta \rangle$. Thus $C$ is $G$-chiral, and so Proposition~\ref{pro:cons} implies that $\G$ has one $G$-orbit of $G$-symmetric and one $G$-orbit of $G$-chiral consistent cycles. 

Suppose now that $G_v \cong \ZZ_2 \times \ZZ_2$ and note that $1$-regularity implies that the action of $G_v$ on its four neighbors is transitive. Let $\vec{C} = (v_0,v_1,\ldots , v_{n-1})$ be a directed $G$-consistent cycle with $v = v_0$. By assumption there exists $\beta \in G_v$ interchanging $v_1$ and $v_{n-1}$. Let $\vec{C'} = \vec{C}\beta$ and observe that both $\vec{C}^{-1}$ and $\vec{C'}$ contain the directed $2$-path $(v_1,v_0,v_{n-1})$. Since $G$ is $1$-regular there exists a unique automorphism $\gamma \in G$ mapping the arc $(v_1,v_0)$ to the arc $(v_0,v_{n-1})$, and so $\vec{C'}$ and $\vec{C}^{-1}$ have the same shunt in $G$ (namely $\gamma$), implying that they coincide. Thus the underlying cycle of $\vec{C}$ is a $G$-symmetric consistent cycle, and consequently all $G$-consistent cycles are $G$-symmetric.  
\hfill $\Box$
\end{proof}

\subsection{Rose Window graphs}
\label{subsec:Rose Window graphs}

In 2008 Wilson~\cite{Wil08} introduced a family of tetravalent graphs now known as the Rose Window graphs. This class of graphs has been studied quite a lot and is now well understood (see for instance~\cite{DobKovMik15, KovKutMar10, KovKutRuf10}). In~\cite{Wil08} Wilson identified four specific subfamilies of Rose Window graphs (defined below) and proved that their members are all arc-transitive. His conjecture that each edge-transitive Rose Window graph (which in the case of Rose Window graphs is equivalent to being arc-transitive) belongs to one of these four subfamilies was confirmed in 2010 by Kov\'acs, Kutnar and Maru\v si\v c~\cite{KovKutMar10}. 

Let $n \geq 3$ be an integer and let $1 \leq r \leq n-1$, with $r \neq n/2$, and $0 \leq a \leq n-1$ be integers. The {\em Rose Window graph} $R_n(a, r)$ is then the graph with vertex-set $\{x_i \mid i\in \ZZ_n \} \cup \{y_i \mid  i\in \ZZ_n\}$ whose edge-set consists of four kinds of edges:
\begin{itemize}
\itemsep = 0pt
\item the set of all {\em rim edges} $x_i x_{i+1}$, $i \in \ZZ_n$;
\item the set of all {\em hub edges} $y_i y_{i+r}$, $i \in \ZZ_n$;
\item the set of all {\em in-spokes} $x_iy_i$, $i \in \ZZ_n$;
\item the set of all {\em out-spokes} $x_iy_{i-a}$, $i \in \ZZ_n$,
\end{itemize}
where all the indices are computed modulo $n$. It is clear that we can assume $a \leq n/2$ and $r < n/2$.
Observe that the graph $R_n(a,r)$ admits the automorphisms
\begin{equation}
\label{eq:rhomu}
\rho = (x_0, x_1, \ldots, x_{n-1})(y_0, y_1, \ldots, y_{n-1}) \quad \mathrm{and} \quad \mu,
\end{equation}
where $\mu$ interchanges each $x_i$ with $x_{n-i}$ and each $y_i$ with $y_{n-i-a}$. Note that the group $\langle \rho, \mu \rangle \cong D_n$ (the dihedral group of order $2n$) has two orbits on the vertex-set of $R_n(a,r)$.

We can now state the result~\cite[Corollary~1.3]{KovKutMar10} giving a complete classification of arc-transitive Rose Window graphs.

\begin{proposition}~{\rm (\cite{KovKutMar10})}
\label{pro:RW}
Let $n \geq 3$ be an integer and let $1 \leq r < n/2$ and $0 \leq a \leq n/2$ be integers. Then the Rose Window graph $R_n(a,r)$ is arc-transitive if and only if it belongs to one of the following four families:
\begin{itemize}
\itemsep = 0pt
\item[(i)] $R_n(2,1)$;
\item[(ii)] $R_{2m}(m-2,m-1)$; 
\item[(iii)] $R_{2m}(2b,r)$, where $b^2 \equiv \pm 1 \pmod{m}$ and either $r = 1$ or $r = m-1$, in which case $m$ must be even.
\item[(iv)] $R_{12m}(3m+2,3m-1)$ or $R_{12m}(3m-2,3m+1)$.
\end{itemize}
\end{proposition}

We remark that in~\cite{KovKutMar10} and \cite{KovKutRuf10} the order of the families (iii) and (iv) (called (d) and (c) there) was reversed but we choose to stick with the order and names given in~\cite{Wil08} where the families were first introduced. 


\section{Maps}
\label{sec:maps}

A {\em map} $\m$ is an embedding of a connected graph $\G$ on a compact surface $S$ without boundary, in such a way that $S\setminus \G$ is a disjoint union of simply connected regions. 
For example, the Platonic Solids can be regarded as maps on the sphere. 
The vertices and edges of the maps are the same as those of its underlying graph, and the {\em faces} of the map are the simply connected regions obtained by removing the graph from the surface.
For simplicity, we often refer to the vertices, edges and faces of a map as their $0$-, $1$-, and $2$-faces, respectively.

By selecting one point in the interior of each $1$- and each $2$-face of $\m$ we can identify an incident triple $\{v,e,f \}$ with the triangle with vertices $v$ and the chosen interior points of the edge $e$ and the face $f$. 
By doing this everywhere on $\m$ we obtain a triangulation of the map, called the {\em barycentric subdivision} ${\mathcal {BS}}(\m)$ of $\m$. 
If the triangles of ${\mathcal {BS}}(\m)$, called {\em flags}, are then in one-to-one correspondence to the incident triples $\{v,e,f\}$, we say that the map $\m$ is {\em polytopal}.  In such a case $\m$ can be regarded as an abstract polytope of rank 3 (in the sense of \cite{ARP}). In this paper we will only be dealing with polytopal maps. We therefore use the term {\em flag} both for the flags themselves and the corresponding incident triples $\{v,e,f\}$ of $\m$. 
We refer the reader to \cite{polymani} for a detailed study of the polytopality of maps and their generalisations to higher dimensions as maniplexes.

For a given flag $\Phi \in {\mathcal {BS}}(\m)$ corresponding to the incident triple $\{v,e,f\}$ we say that $v$ is the {\em vertex}, $e$ is the {\em edge} and $f$ is the {\em face} of $\Phi$, respectively, and that $v$, $e$ and $f$ {\em belong} to $\Phi$. It is convenient to colour the vertices of ${\mathcal {BS}}(\m)$ with the colours $0$, $1$ and $2$, whenever they represent a vertex, edge or face of $\m$, respectively. 
Observe that given a flag $\Phi \in {\mathcal{BS}}(\m)$, it shares its three sides with three other flags of ${\mathcal {BS}}(\m)$, that we shall denote by $\Phi^0$, $\Phi^1$ and $\Phi^2$, where $\Phi$ and $\Phi^i$ share the vertices of colours different from $i$. (see Figure~\ref{fig:adj_flags}). 
 The flags $\Phi$ and $\Phi^i$ are said to be {\em $i$-adjacent flags}.
 \begin{figure}[htbp]
\begin{center}
\includegraphics[width=5cm]{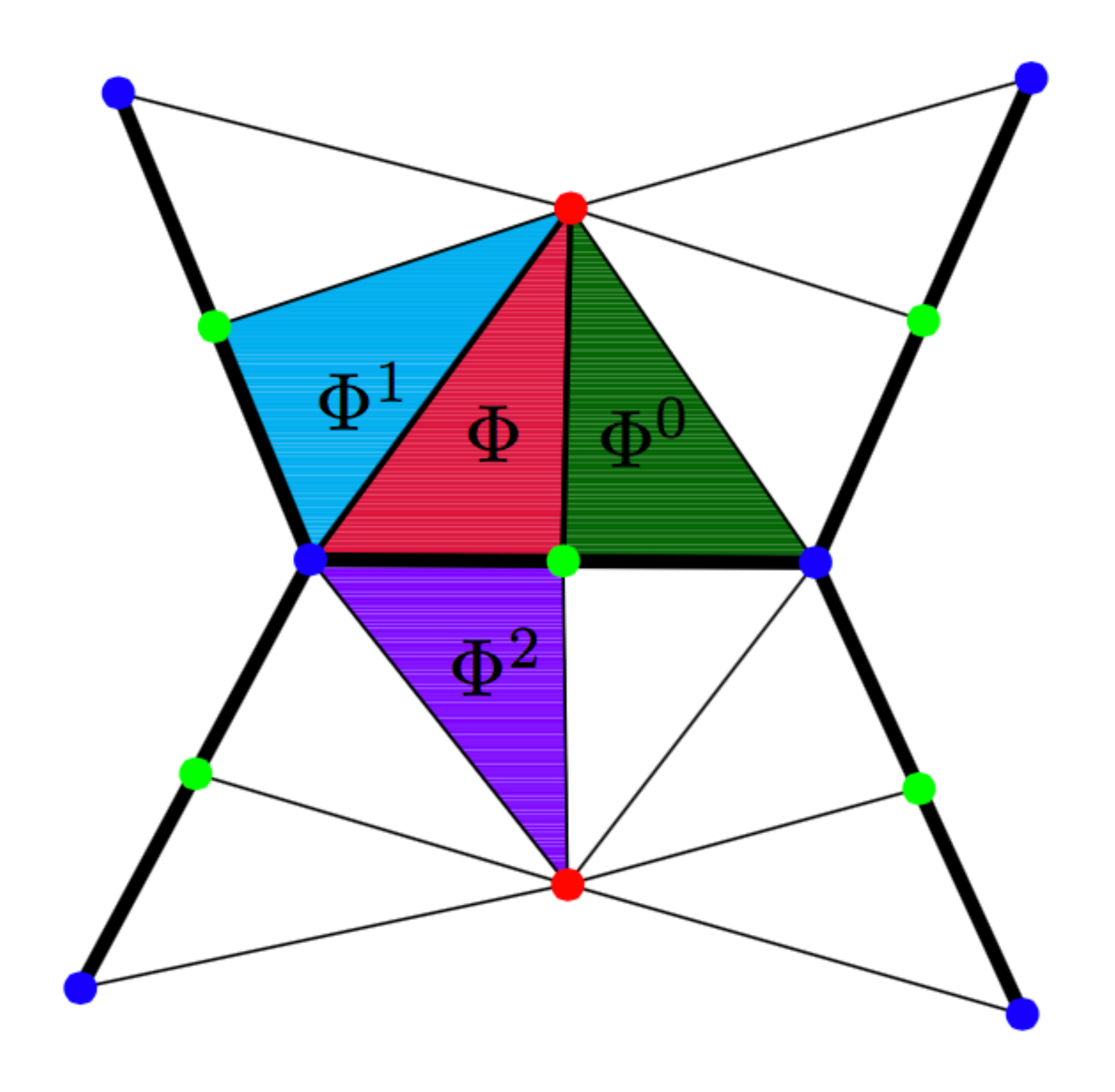}
\caption{A flag $\Phi$ of a map and its adjacent flags.}
\label{fig:adj_flags}
\end{center}
\end{figure}

Given a sequence $i_0, i_1, \dots, i_k$, with $i_j \in \{0,1,2\}$, we define inductively the flag $\Phi^{i_0, i_1, \dots, i_k}$ as the $i_{k}$-adjacent flag to the flag $\Phi^{i_0, i_1, \dots, i_{k-1}}$.
Note that as each edge of $\m$ belongs to exactly four flags, we have that $\Phi^{0,2} = \Phi^{2,0}$ holds for every flag $\Phi\in \fl(\m)$, where $\fl(\m)$ denotes the set of all flags of $\m$. Note also that since the underlying graph $\G$ of $\m$ is connected, given any two flags $\Phi, \Psi \in \fl(\m)$ there exists a sequence $i_0, i_1, \dots, i_k$, with $i_j \in \{0,1,2\}$ such that $\Psi = \Phi^{i_0, i_1, \dots, i_k}$.

Whenever a map is such that all its faces have length $p$ and all its vertices have valency $q$, the map is said to be {\em equivelar} and it has {\em type} $\{p,q\}$.

An {\em automorphism} of a map $\m$ is an automorphism of its underlying graph $\G$ that also preserves the faces of $\m$. The group of all automorphisms of $\m$, denoted by $\Aut(\m)$, has a natural action on the set of flags of $\m$. The following straightforward observation can be very useful.

\begin{lemma}
\label{le:i_adj_autom}
For each automorphism $\alpha \in \Aut(\m)$, each flag $\Phi \in \fl(\m)$ and each $i \in \{0,1,2\}$ the $i$-adjacent flag of $\Phi\alpha$ is the $\alpha$-image of the $i$-adjacent flag of $\Phi$, that is $\Phi^i \alpha = (\Phi\alpha)^i$. 
\end{lemma}

Hence, the connectivity of $\m$ implies that the action of $\Aut(\m)$ is free on $\fl(\m)$ (that is, no nonidentity automorphism of $\m$ fixes any flag). This implies that the action of an automorphism is completely determined by its action on any given flag. Lemma~\ref{le:i_adj_autom} also implies the following.

\begin{lemma}\label{adjacentaction}
Let $\m$ be a map and let ${\cal O}_1$ and ${\cal O}_2$ be (possibly the same) orbits of the action of $\Aut(\m)$ on $\fl(\m)$. Suppose that $\Phi^i\in{\cal O}_2$ holds for some $\Phi \in {\cal O}_1$ and some $i \in \{0,1,2\}$. Then $\Psi^i \in {\cal O}_2$ for every $\Psi \in {\cal O}_1$.
\end{lemma}

We shall say that an automorphism $\alpha$ of $\m$ is a {\em reflection} whenever there exists $\Phi\in\fl(\m)$ and $i\in\{0,1,2\}$ such that $\Phi\alpha=\Phi^i$. Note that by Lemma~\ref{le:i_adj_autom} each reflection is an involution.
We say that $\alpha$ is a {\em  one step rotation} at the face (resp. at the vertex) of $\Phi$ if $\Phi \alpha = \Phi^{0,1}$ or $\Phi \alpha = \Phi^{1,0}$ (resp. $\Phi \alpha = \Phi^{2,1}$ or $\Phi \alpha = \Phi^{1,2}$). Note that since the action of $\Aut(\m)$ on $\fl(\m)$ is free there can be at most one pair of mutually inverse one step rotations at any given vertex or face. The following observation is straightforward.

\begin{lemma}\label{face-edge}
Let $\m$ be a map such that the underlying graph has no multiple edges or loops and has minimal degree at least $3$. If for some face $f$ of $\m$ there exists a one step rotation at $f$ in $\Aut(\m)$ then the traversal of $f$ along its boundary visits each of its vertices and edges exactly once. 
\end{lemma}

Lemma~\ref{face-edge} thus implies that maps with underlying simple graphs of minimal degree at least $3$ that admit a one step rotation at every face are polytopal.
\medskip

If the action of $\Aut(\m)$ has $k$ orbits on the flags of $\m$ we say that $\m$ is a {\em $k$-orbit map}. The $1$-orbit maps are usually called {\em reflexible} in the literature. A map is reflexible if and only if given a {\em base} flag $\Phi$ there exist automorphisms $\alpha_i$, $i\in\{0,1,2\}$, sending $\Phi$ to $\Phi^i$. In such a case, the automorphism group of $\m$ is generated by $\alpha_0, \alpha_1$ and $\alpha_2$.

The $2$-orbit maps have been studied, for example, in \cite{ruiphd} and \cite{2-orbit}. There exist $7$ classes of $2$-orbit maps. Given $I \subsetneq \{0,1,2\}$ we say that a $2$-orbit map is in {\em class $2_I$} if for any given flag $\Phi$ we have that $\Phi^i$ is in the same $\Aut(\m)$-orbit as $\Phi$ if and only if $i \in I$. 
 By Lemma~\ref{adjacentaction} and the fact that we are dealing with 2-orbit maps, this definition does not depend on the choice of the flag $\Phi$, and thus the $7$ classes are disjoint.
We abbreviate $2_\emptyset$ by $2$ and $2_{\{i\}}$ by $2_i$ for each of $i \in \{0,1,2\}$. 

Maps in class $2$ correspond to {\em chiral} maps, that is, maps that have two orbits on flags under the automorphism group in such a way that adjacent flags belong to different orbits. A map that is either reflexible or chiral is called a {\em rotary} map. Rotary maps admit one step rotations around each of its faces and each of its vertices.

A map is said to be {\em $j$-face transitive} if its automorphism group acts transitively on the $j$-faces and it is called {\em fully transitive} if it is $j$-face transitive for all $j \in \{0,1,2\}$.
Rotary maps are examples of fully transitive maps. However, this is no longer true for all $2$-orbit maps. 
In fact, a $2$-orbit map is $j$-face transitive, for some $j \in \{0,1,2\}$, if and only if $I \neq \{0,1,2\} \setminus \{j\}$ (see \cite[Theorem~5]{2-orbit}). In particular this means that for each $j \in \{0,1,2\}$ there is exactly one class of $2$-orbit maps that are not $j$-face transitive and hence there are just three classes of $2$-orbit maps that are not fully transitive.

A {\em Petrie polygon} of a map $\m$ is a path $P$ along the edges of the underlying graph $\G$ such that every two consecutive edges of $P$ are consecutive edges on the same face of $\m$, but no three consecutive edges of $P$ are consecutive edges on the same face of $\m$. 
For example, a Petrie polygon of a tetrahedron contains each of the four vertices of the tetrahedron exactly once. It is not difficult to see that every edge of $\m$ belongs to at most two Petrie polygons. 
The map having $\G$ as the underlying graph and all the Petrie polygons of $\m$ as faces is called the {\em Petrial} or {\em Petrie dual} of $\m$, and shall be denoted by $\m^\pi$. 
It is well known (see, for example, \cite{medial}) that $(\m^\pi)^\pi \cong \m$ and that $\Aut(\m) \cong \Aut(\m^\pi)$.
Hence, the Petrial of a $k$-orbit map is again a $k$-orbit map. 
Moreover, if $\m$ is a $2$-orbit map in class $2_I$, for some $I \subsetneq \{0,1,2\}$, then $\m^\pi$ is in class $2_{I'}$, where $I' = I \setminus \{0\}$, if $0\in I$ and $2\notin I$, $I'=I\cup \{0\}$ if $0,2\notin I$ and $I'=I$ if $2\in I$ (see \cite{medial}).

%

\subsection{Arc-transitive maps}

A map is said to be {\em arc-transitive} if its automorphism group is arc-transitive on the underlying graph.
Rotary maps are examples of arc-transitive maps.

Let $\m$ be an arc-transitive map and let $\Phi = (v,e,f) \in \fl(\m)$. If $u$ is the other vertex of the edge $e$, then there exists $\alpha \in \Aut(\m)$, fixing the edge $e$ while interchanging $u$ and $v$. 
This means that $\alpha$ sends $\Phi$ to either $\Phi^0$ or $\Phi^{0,2}$. 
Consequently, $\alpha$ sends $\Phi^2$ to either $\Phi^{0,2}$ or $\Phi^{0}$.
Since $e$ can be mapped by an automorphism of $\Aut(\m)$ to any other edge of $\m$, the connectivity of $\m$ and Lemma~\ref{adjacentaction} imply that there are at most two orbits of flags under the action of $\Aut(\m)$.

Of course, if the action of $\Aut(\m)$ on $\fl(\m)$ has exactly one orbit, then $\m$ is reflexible. However if there are two orbits, then the above remarks imply that $\m$ has to be in class $2_I$ for some $I \subset \{0,1,2\}$ with $2 \notin I$, that is, $\m$ belongs to one of the following four classes of maps: $2$, $2_0$, $2_1$ or $2_{\{0,1\}}$.

The reflexible and chiral maps have been extensively studied in the literature. 
Moreover the maps in class $2_0$ are related to the chiral ones via the Petrie dual operation (and their graphs and groups are the same). This leaves us with the maps of classes $2_1$ and $2_{\{0,1\}}$ which are also related via the Petrie operator. In this paper we thus restrict ourselves to maps of class $\2$. 

Maps in class $\2$ are {\em hereditary} in the sense that all the combinatorial symmetries of their faces can be extended to the entire map (see \cite{hered} for a study of hereditary polytopes). 
We remark that these maps are of type $2^*$ in the sense of \cite{edgetrans}.

If $\m$ is a map in class $\2$, then for each flag $\Phi$ there exist (unique) automorphisms $\alpha_0(\Phi)$ and $\alpha_1(\Phi)$ sending $\Phi$ to the flags $\Phi^0$ and $\Phi^1$, respectively. By Lemma~\ref{le:i_adj_autom} the automorphism $\alpha_1(\Phi^2)$ maps $\Phi$ to $\Phi^{2,1,2}$, and so there also exists an (unique) automorphism $\alpha_{212}(\Phi)$ of $\m$ sending a given flag $\Phi$ to $\Phi^{2,1,2}$. Whenever the flag $\Phi$ will be clear from the context we will write $\alpha_0$, $\alpha_1$ and $\alpha_{212}$ instead of $\alpha_0(\Phi)$, $\alpha_1(\Phi)$ and $\alpha_{212}(\Phi)$, respectively.
 In \cite{2-orbit}, it was shown that if $\m$ is a map in class $2_{\{0,1\}}$ and $\Phi$ is any flag of $\m$, then $\Aut(\m) = \langle \alpha_0(\Phi), \alpha_1(\Phi), \alpha_{212}(\Phi)\rangle$.

Recall that maps in class $\2$ are not $2$-face transitive. On the other hand, in view of the existence of automorphisms $\alpha_0$ and $\alpha_1$, all of the flags corresponding to a given face of such a map $\m$ are in the same $\Aut(\m)$-orbit. Since there are only two orbits of flags this implies that there are two orbits of faces. In other words, there is no automorphism of $\m$ mapping a flag $\Phi$ to either $\Phi^2$ or to $\Phi^{2,0}$. Since the action of $\Aut(\m)$ on $\fl(\m)$ is free, the group $\Aut(\m)$ acts as a $1$-regular group on the underlying graph of $\m$. Moreover, vertex-transitivity of these maps implies that both orbits of faces occur at every vertex, and as a flag cannot be sent to its $2$-adjacent flag, the members of the two orbits of faces alternate around a vertex. In particular, the valency of the vertices is even. 
If $n$ and $m$ are the lengths of the faces of the map $\m$ from the two orbits and the valency of each vertex is $2q$, then we say that the {\em type} of $\m$ is $\left\{ \begin{array}{c} n \\ m \end{array}, 2q\right\}$. In this case the face stabilizers $\langle \alpha_0, \alpha_1 \rangle$ and $\langle \alpha_0, \alpha_{212}\rangle$ are isomorphic to $D_n$ and $D_m$, respectively (assuming the base flag $\Phi$ is in a face of length $n$), while the vertex stabilizer $\langle \alpha_1, \alpha_{212} \rangle$ is isomorphic to $D_q$. 
Furthermore $\langle \alpha_1, \alpha_{212} \rangle$ acts transitively (in fact regularly) on the neighbours of the {\em base} vertex belonging to $\Phi$.

The above remarks imply that the smallest admissible valency of the underlying graph of a map in class $\2$ is $4$. It thus seems natural to first study such maps. It is the aim of this paper to initiate such an investigation. In particular, we classify all maps in class $\2$, underlying an arc-transitive Rose Window graph. We remark that the rotary maps underlying these graphs were classified in~\cite{KovKutRuf10}.
\medskip

 In order to give the desired classification, we first need some results about maps in class $2_{\{0,1\}}$ and their underlying graphs.
Following \cite{STG}, the next lemma is straightforward.

\begin{lemma}
\label{le:aux}
If $\m$ is a map such that there exists a flag $\Phi$ and automorphisms $\alpha_0$, $\alpha_1$ and $\alpha_{212}$ sending $\Phi$ to $\Phi^0$, $\Phi^1$ and $\Phi^{2,1,2}$, respectively, then $\m$ is either a $2$-orbit map in class $\2$ or it is a reflexible map.
\end{lemma}

Let $\m$ be a map in class $\2$ and let $\alpha_0$, $\alpha_1$ and $\alpha_{212}$ be the {\em distinguished generators} of $\Aut(\m)$ with respect to some base flag $\Phi$. Then the automorphism $\alpha_0\alpha_1$ acts as a $1$-step rotation of the face belonging to $\Phi$ and $\alpha_0\alpha_{212}$ acts as a $1$-step rotation of the face belonging to $\Phi^2$. 
Moreover,  $\alpha_0$ reveres both of these faces. We have therefore established the following lemma.

\begin{lemma}
\label{facesofhereditarymaps}
Let $\m$ be a map in class $\2$ with the underlying graph $\G$. Then the boundaries of faces of $\m$ are $\Aut(\m)$-symmetric consistent cycles of $\G$. 
\end{lemma}

\begin{corollary}
\label{cor:twoedges} 
Let $\m$ be a map in class $\2$. Then no two faces of $\m$ share two consecutive edges.
\end{corollary}
 
\begin{proof}
Since $\Aut(\m)$ is $1$-regular on the underlying graph, each face has an unique pair of mutually inverse one step rotations in $\Aut(\m)$. Hence, if two faces share two common consecutive edges the one step rotations for both faces coincide, implying that they share all its vertices and edges. But in this case the map is a reflexible map of type $\{p,2\}$ on the sphere, a contradiction.
\hfill $\Box$
\end{proof}
\medskip

We finish this section with a result that is of great help when dealing with maps in class $\2$ with underlying tetravalent graphs. 

\begin{theorem}
\label{the:1-regular}
Let $\G$ be a tetravalent graph admitting a $1$-regular group of automorphisms $G$. If $\G$ is the underlying graph of a map $\m$ in class $\2$ with $\Aut(\m) = G$ then all orbits of $G$-consistent cycles of $\G$ are $G$-symmetric. Moreover, if $G = \Aut(\G)$ and all orbits of $G$-consistent cycles of $\G$ are $G$-symmetric then for any two orbits of $G$-consistent cycles of $\G$ there exists a map $\m$ in class $\2$ with $\Aut(\m) = G$ and underlying graph $\G$ such that the boundaries of its faces are the members of these two orbits.
\end{theorem}

\begin{proof}
By Proposition~\ref{pro:cons} the graph $\G$ has three orbits of $G$-consistent directed cycles and at least one of the orbits of $G$-consistent cycles is $G$-symmetric. 

Let us start by assuming that $\G$ is the underlying graph of a map $\m$ in class $\2$ with $\Aut(\m) = G$. By Lemma~\ref{facesofhereditarymaps} the faces of $\m$ are $G$-symmetric consistent cycles, and since $G$ has two orbits on the set of faces of $\m$ this shows that $\G$ has at least two orbits of $G$-symmetric consistent cycles, implying that $\Gamma$ in fact has three orbits of $G$-consistent cycles, all of which are $G$-symmetric. 

For the second part of the theorem suppose $G = \Aut(\G)$ and that all orbits of $G$-consistent cycles are $G$-symmetric. By Proposition~\ref{pro:cons} there are three of them; let ${\cal O}_1$ and ${\cal O}_2$ be any two. We now show that there is a map in class $\2$ having the elements of ${\cal O}_1$ and ${\cal O}_2$ as faces. In fact, by Lemma~\ref{le:cons}, taking all the elements of both ${\cal O}_1$ and ${\cal O}_2$ as faces, we do get a map $\m$ with underlying graph $\G$. We just need to show that this map $\m$ is in class $\2$. Since we have chosen two complete orbits of $G$-consistent cycles as the faces of $\m$, every automorphism of $\G$ sends faces to faces, implying that it is an automorphism of $\m$, and so $\Aut(\G) = G = \Aut(\m)$. Moreover, since $G$ is $1$-regular, the map $\m$ is not reflexible. By Lemma~\ref{le:aux} we thus only need to show that we can send a given flag $\Phi$ to the flags $\Phi^0$, $\Phi^1$ and $\Phi^{2,1,2}$.

Let $\Phi=(v,e,f_1)$ be a flag of $\m$ and let $u$ be the other vertex of $\G$, incident to $e$. Without loss of generality assume that $f_1$ belongs to ${\cal O}_1$. Denote by $f_2$ the unique (confront Lemma~\ref{le:cons}) face from ${\cal O}_2$ containing $e$. First, since $f_1$ is a $G$-symmetric consistent cycle there exists $\alpha_0 \in G$ fixing $f_1$ and $e$ and interchanging $u$ and $v$. We can therefore map $\Phi$ to its $0$-adjacent flag $\Phi^0$. Let $\beta, \beta' \in G$ be the shunts of $f_1$ and $f_2$, respectively, that send $v$ to $u$. Then $\alpha_1 := \beta\alpha_0$ maps $\Phi$ to $\Phi^1$. Finally, $\alpha_{212} := \beta'\alpha_0$ maps $\Phi$ to $\Phi^{2,1,2}$, proving that $\m$ is indeed in class $\2$.
\hfill $\Box$
\end{proof}
\bigskip

Combining together Lemma~\ref{le:tetra_cons} and Theorem~\ref{the:1-regular} we have the following useful corollary.

\begin{corollary}
\label{cor:1-regular}
Let $\G$ be a tetravalent graph with a $1$-regular group of automorphisms. Then $\G$ is the underlying graph of a map of class $\2$ if and only if the vertex stabilizers in $\Aut(\G)$ are isomorphic to the Klein $4$-group. Moreover, in this case there are three pairwise nonisomorphic such maps. 
\end{corollary}


\section{Maps of class $\2$ with underlying Rose Window graphs}
\label{sec:2_01onRW}

As announced in the previous section we now classify all maps of class~$\2$ whose underlying graph is a Rose Window graph. Since the chiral maps underlying Rose Window graphs have already been classified in~\cite{KovKutRuf10} the remarks from the previous section imply that this completes the classification of all arc-transitive maps corresponding to Rose Window graphs.
We analyze each of the four subfamilies from Proposition~\ref{pro:RW} in a separate subsection.


\subsection{Family (i)}
\label{subsec:RW_i}

We start by considering the family of graphs $R_n(2,1)$, $n \geq 3$. The graph $R_n(2,1)$ is isomorphic to the lexicographic product $C_n[2K_1]$ of a cycle of length $n$ with two independent vertices and is known also as the wreath graph. These graphs appear in various investigations of symmetries of graphs and have thus been studied in great detail before (see for instance~\cite{PraXu89} where certain generalizations, now know as the Praeger-Xu graphs, have been introduced and their automorphism groups determined) . 

Throughout this subsection let $\G = R_n(2,1)$. For convenience we relabel the vertices of $\G$ in the following way. For each $i \in \ZZ_n$ we let $u_i = x_i$ and $v_i = y_{i-1}$. With this notation each pair of vertices $u_i$ and $v_i$ have the same neighborhood $\{u_{i\pm1}, v_{i\pm1}\}$. The permutations $\rho$ and $\mu$ from (\ref{eq:rhomu}) thus map in such a way that $u_i\rho = u_{i+1}$, $v_i\rho = v_{i+1}$, $u_i\mu = u_{-i}$ and $v_i\mu = v_{-i}$ for all $i \in \ZZ_n$. For each $i \in \ZZ_n$ let $\sigma_i$ be the involution interchanging $u_i$ and $v_i$ and fixing all other vertices. Clearly $\sigma_i = \rho^{-i}\sigma_0\rho^i$  and $\sigma_i\sigma_j = \sigma_j\sigma_i$ hold for all $i,j \in \ZZ_n$. Moreover, it is well known that, unless $n = 4$ in which case $\G \cong K_{4,4}$, the 2-element sets $\{u_i,v_i\}$ are blocks of imprimitivity for $\Aut(\G) = \langle \rho, \mu, \sigma_0\rangle$ which is thus of order $n2^{n+1}$ with $N = \langle \sigma_0, \sigma_1, \ldots , \sigma_{n-1}\rangle \triangleleft \Aut(\G)$.
\medskip

Since $R_4(2,1) \cong K_{4,4}$ is a bit special, we deal with it separately. Using a suitable computer package it is easy to see that there is exactly one map of class $\2$ with $K_{4,4}$ as its underlying graph; the faces are of lengths $4$ and $8$.
For the rest of this subsection we thus assume $n \neq 4$. Our approach in determining all maps of class~$\2$ on $\G$ is similar to the one taken in~\cite{KovKutRuf10}. We first determine the structure of potential $1$-regular subgroups of $\Aut(\G)$ which could be the automorphism groups of such maps.

Suppose then that $\m$ is a class~$\2$ map with the underlying graph $\G$ and let $T = N \cap \Aut(\m)$. Clearly, each $\sigma \in N$ can uniquely be expressed as $\sigma = \Pi_{j=0}^{n-1}\sigma_j^{i_j}$, where $i_j \in \{0,1\}$, and so we can denote each such $\sigma$ with the corresponding $n$-tuple $(i_0, i_1, \ldots , i_{n-1})$. 

\begin{lemma}
\label{le:T(i)}
We either have 
\begin{equation}
\label{eq:T1}
	\begin{array}{c} T = \{(0,0,\ldots , 0), (0,1,1,0,1,1,\ldots , 0,1,1),\\
	 (1,0,1,1,0,1,\ldots , 1,0,1), (1,1,0,1,1,0,\ldots , 1,1,0)\},\end{array}
\end{equation}
in which case $3\mid n$, or
\begin{equation}
\label{eq:T2}
	\begin{array}{c} T = \{(0,0,\ldots , 0), (0,1,0,1,\ldots , 0,1),\\
	 (1,0,1,0,\ldots , 1,0), (1,1,\ldots , 1)\},\end{array}
\end{equation}
in which case $2\mid n$. In particular, $\gcd(n,6) \neq 1$.
\end{lemma}

\begin{proof}
Since $\Aut(\m)$ is $1$-regular it easily follows that $T \cong \ZZ_2 \times \ZZ_2$. Let $T=\{1, t_1,t_2,t_3\}$. The subgroup $N$ is normal in $\Aut(\G)$, implying that $T$ is normal in $\Aut(\m)$. Since $\Aut(\m)$ is arc-transitive on $\G$ and the sets $\{u_i, v_i\}$ are blocks of imprimitivity for $\Aut(\m)$, the group $\Aut(\m)$ contains elements of the form $\rho \sigma$ and  $ \mu \sigma'$, where $\sigma, \sigma' \in N$. Since $T$ is normal in $\Aut(\m)$ and $N$ is abelian we thus get $T = T^{\sigma \rho^{-1}} = T^{ \rho^{-1}}$ and  $T = T^{\sigma'\mu^{-1}} = T^{ \mu^{-1}}$, implying that $T^{\rho} = T = T^{\mu}$. 

Now, $1$-regularity of $\Aut(\m)$ implies that for any $1 \neq t = (i_0, i_1, \ldots , i_{n-1}) \in T$ we cannot have $i_j = i_{j+1} = 0$ for any $j \in \ZZ_n$. We can thus assume that $t_1=(0,1,i_2, i_3, \ldots, i_{n-3}, i_{n-2}, 1)$. Then $t_1^{\rho} = \rho^{-1}t_1\rho = (1,0,1,i_2,i_3, \ldots,i_{n-3}, i_{n-2}) \neq t_1$. We can assume $t_1^{\rho} = t_2$. Now, $t_2^{\rho}= (i_{n-2},1,0,1,i_2, i_3, \ldots,i_{n-3}) \neq t_2$, and so we either have $t_2^{\rho} = t_3$ (in which case $t_3^\rho = t_1$) or $t_2^{\rho}=t_1$ (in which case $t_3^\rho = t_3$). It is now clear that in the first case $n$ is divisible by $3$ and $T$ is as in (\ref{eq:T1}) and in the second case $n$ is even and $T$ is as in (\ref{eq:T2}).
\hfill $\Box$
\end{proof}
\bigskip

We can now analyze the different possibilities for the faces of our map $\m$. By Proposition~\ref{pro:cons} the graph $\G$ has three orbits of $\Aut(\G)$-consistent cycles. The representatives of the orbits are $(u_0,u_1,v_0,v_1)$, $(u_0,u_1, \ldots , u_{n-1})$ and $(u_0,u_1, \ldots , u_{n-1}, v_0, v_1, \ldots , v_{n-1})$, with shunts $\sigma_1\mu\rho, \rho$ and $\rho\sigma_0$, respectively. Thus, all the $\Aut(\G)$-consistent cycles are clearly $\Aut(\G)$-symmetric. By Lemma~\ref{facesofhereditarymaps} the possible face lengths for $\m$ are $4$, $n$ and $2n$. Moreover, the following holds.

\begin{lemma}
\label{le:lengths}
The map $\m$ has faces of two different lengths.
\end{lemma}

\begin{proof}
Recall that $\m$, being a map of class $\2$, admits a one-step rotation around each of its faces. Thus, by Lemma~\ref{face-edge}, each edge is on the boundary of two different faces of $\m$. Since $(u_i,u_{i+1},v_i,v_{i+1})$ is clearly the only $\Aut(\G)$-consistent $4$-cycle containing any of the corresponding four edges (recall that $n \neq 4$), it is clear that each edge of $\G$ lies on at least one face of length greater than $4$. 

Suppose $f$ is a face of $\m$ of length $n$. Since $|T|=4$ and $f$ clearly has exactly one vertex of each block of imprimitivity $B_i = \{u_i, v_i\}$ the set ${\cal O}_1=fT=\{ft\mid t\in T\}$ is the orbit of $f$ under the action of $\Aut(\m)$. By way of contradiction suppose that the other $\Aut(\m)$-orbit of faces of $\m$ also consists of faces of length $n$ and let $f'$ be any face from the second orbit ${\cal O}_2 = f'T$. Since $f'$ corresponds to an $\Aut(\G)$-consistent cycle it also contains exactly one vertex from each of the blocks $B_i$, and so there exists $\sigma \in N$ such that $f' = f\sigma$. But as $N$ is abelian we get ${\cal O}_2 = f'T = f\sigma T = fT\sigma = {\cal O}_1\sigma$, and so $\sigma$ interchanges the two orbits of faces of $\m$, implying that it is in fact an automorphism of $\m$. But then $\sigma \in T$, and so ${\cal O}_2 ={\cal O}_1\sigma = {\cal O}_1$, a contradiction. 

A similar argument shows that $\m$ also cannot have all faces of length $2n$.
\hfill $\Box$
\end{proof}
\bigskip

The analysis of possible maps $\m$ of class $\2$ whose underlying graph is $\G$ is now straightforward. By Lemma~\ref{le:lengths} we either have an orbit of faces of length $n$ (and either an orbit of faces of length $4$ or faces of length $2n$) or one orbit of faces of length $2n$ and one orbit of faces of length $4$. Moreover,  Corollary~\ref{cor:twoedges} implies that the map is completely determined once we have chosen one orbit of faces of length $n$ or $2n$ and decided on the length of the faces from the other orbit. Next, in view of the action of $\Aut(\G)$ and the remarks from the paragraph preceding Lemma~\ref{le:lengths} we can assume that, in the case that $\m$ has faces of length $n$, one of them is $f = (u_0,u_1, \ldots , u_{n-1})$ while in the case it does not have faces of length $n$ one of the faces of length $2n$ is $f = (u_0,u_1, \ldots, u_{n-1},v_0,v_1, \ldots,v_{n-1})$. The corresponding $\Aut(\m)$ orbit of $f$ is then completely determined by the action of $T$, which by Lemma~\ref{le:T(i)} is also known (up to the two possibilities). Once the faces have been determined one only needs to check that we indeed have the required automorphisms of the map for Lemma~\ref{le:aux} to apply. Note that, since the faces are of two different lengths, the obtained map cannot be reflexible, and is thus automatically of class $\2$.

\begin{theorem}
\label{th:RW(i)}
Let $\G = R_n(2,1)$ be a Rose Window graph with $n \geq 3$. Then $\G$ is the underlying graph of a map $\m$ of class $\2$ if and only if $\gcd(n,6) \neq 1$. Moreover, letting $n_0 \in \{0,2,3,4,6,8,9,10\}$ be the residue of $n$ modulo $12$ the following holds:
\begin{itemize}\itemsep = 0pt
\item[(i)] if $n = 4$, then $\G$ is the underlying graph of exactly one map of class~$\2$ with face lengths $4$ and $8$.
\item[(ii)] if $n_0 \in \{3,9\}$, then $\G$ is the underlying graph of a unique map in class $\2$; the faces are of lengths $4$ and $n$.
\item[(iii)] if $n_0 \in \{4,8\}$, then $\G$ is the underlying graph of two nonisomorphic maps of class~$\2$; one has faces of lengths $4$ and $n$ and the other has faces of lengths $4$ and $2n$.
\item[(iv)] if $n_0 \in \{2,10\}$, then $\G$ is the underlying graph of three nonisomorphic maps of class~$\2$; one has faces of lengths $4$ and $n$, one has faces of lengths $4$ and $2n$, and one has faces of lengths $n$ and $2n$.
\item[(v)] if $n_0 =0$, then $\G$ is the underlying graph of three nonisomorphic maps of class~$\2$; two have faces of lengths $4$ and $n$, and one has faces of lengths $4$ and $2n$.
\item[(vi)] if $n_0=6$, then $\G$ is the underlying graph of four nonisomorphic maps of class~$\2$; two have faces of lengths $4$ and $n$, one has faces of lengths $4$ and $2n$, and one has faces of lengths $n$ and $2n$.
\end{itemize}
\end{theorem}

\begin{proof}
The case $n = 4$ has been dealt with at the beginning of this section. For the rest of the proof we thus assume $n \neq 4$. By Lemma~\ref{le:T(i)} at least one of $2$ and $3$ must divide $n$. We now consider the possibilities for the combinations of the lengths of faces of $\m$. We first analyze the possibility that $\m$ has faces of length $n$. Recall that we can assume that one of the $n$-faces is $f = (u_0,u_1, \ldots , u_{n-1})$. We now separate the argument for the two possibilities regarding the subgroup $T$. 

Suppose first that $T$ is as in (\ref{eq:T1}) and recall that in this case $3$ divides $n$. The four faces of length $n$ are then (see Figure~\ref{fig:family1-nT1}):
$$ 
\begin{array}{rcl} f &=& (u_0, u_1, \ldots , u_{n-1}), \\
	ft_1 &= &(u_0, v_1, v_2, u_3, v_4, v_5, \ldots , u_{n-3}, v_{n-2}, v_{n-1}),\\
ft_2 &=& (v_0, u_1, v_2, v_3, u_4, v_5, \ldots , v_{n-3}, u_{n-2}, v_{n-1})\ \text{and} \\
ft_3 &=& (v_0, v_1, u_2, v_3, v_4, u_5, \ldots , v_{n-3}, v_{n-2}, u_{n-1}).\end{array}
$$\begin{figure}[!ht]
  \centering
  \includegraphics[scale =0.32]{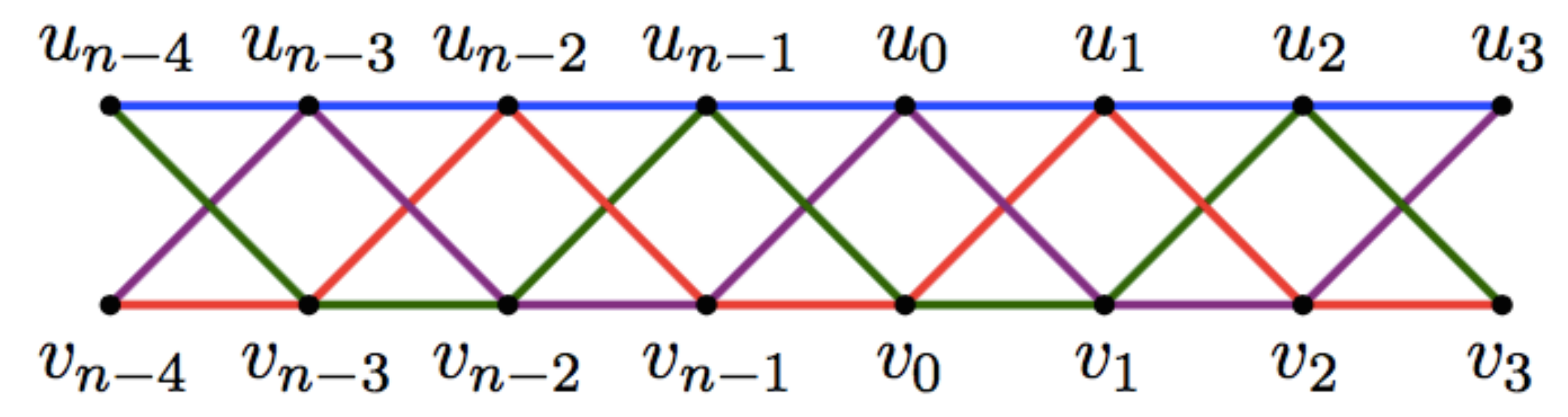}
 \caption{The faces of length $n$ in the case that $T$ is as in~(\ref{eq:T1}).} \label{fig:family1-nT1}
\end{figure}

Let $f'$ be the face containing the edge $u_0u_1$, different from $f$. If it is not a $4$-cycle then the fact that $3 \mid n$ and Corollary~\ref{cor:twoedges} imply that it is $(u_0, u_1, v_2, u_3, u_4, v_5, \ldots , u_{n-3}, u_{n-2}, v_{n-1})$, contradicting Lemma~\ref{le:lengths}. Thus the non $n$-faces of $\m$ are the $4$-cycles $(u_i, u_{i+1}, v_i, v_{i+1})$, $i \in \ZZ_n$. It remains to be shown that the resulting map $\m$ is indeed a map of class $\2$. It is clear that $\rho, \mu \in \Aut(\m)$. Let $\Phi$ be the flag corresponding to the vertex $u_0$, edge $u_0u_1$ and the $4$-face $(u_0,u_1,v_0,v_1)$. Then $\mu \rho = \alpha_0$, $t_1 = \alpha_1$ and $\mu = \alpha_{212}$, and so Lemma~\ref{le:aux} implies that $\m$ is a map of class~$\2$.

Suppose now that $T$ is as in (\ref{eq:T2}) and recall that in this case $n$ is even. The four faces of length $n$ are then (see Figure~\ref{fig:family1-nT2}):
$$ 
\begin{array}{rcl} f &=& (u_0, u_1, \ldots , u_{n-1}), \\
	ft_1 &=& (u_0, v_1, u_2, v_3, \ldots , u_{n-2}, v_{n-1}),\\
	ft_2 &=& (v_0, u_1, v_2, u_3, \ldots , v_{n-2}, u_{n-1})\ \text{and} \\
	ft_3 &=& (v_0, v_1, \ldots , v_{n-1}).\end{array}
$$
\begin{figure}[h!]
  \centering
  \includegraphics[scale =0.32]{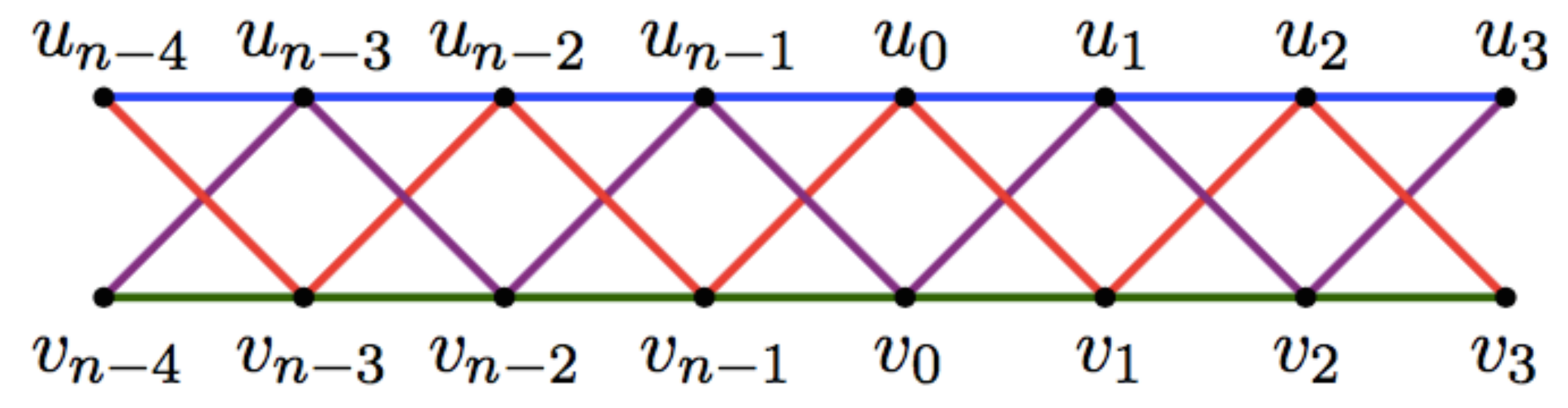}
 \caption{The faces of length $n$ in the case that $T$ is as in~(\ref{eq:T2}).} \label{fig:family1-nT2}
\end{figure}

Let $f'$ be the face containing the edge $u_0u_1$, different from $f$. If it is not a $4$-face then its boundary contains the path $(u_0, u_1, v_2, v_3, u_4, u_5, \dots , u_{n-3}, v_{n-2}, v_{n-1})$, and so Lemma~\ref{le:lengths} implies that $n \equiv 2 \pmod 4$. The two $2n$-faces are thus
$$
\begin{array}{l} f' = (u_0, u_1, v_2, v_3, \ldots , u_{n-2}, u_{n-1}, v_0, v_1, \ldots , u_{n-4}, u_{n-3}, v_{n-2}, v_{n-1})\ \text{and} \\
	f't_1 = (u_0, v_1, v_2, u_3, u_4, \ldots , v_{n-1}, v_0, u_1, u_2, \ldots , v_{n-3}, v_{n-2}, u_{n-1}).\end{array}
$$
\begin{figure}[htbp]
\begin{center}
\includegraphics[scale =0.32]{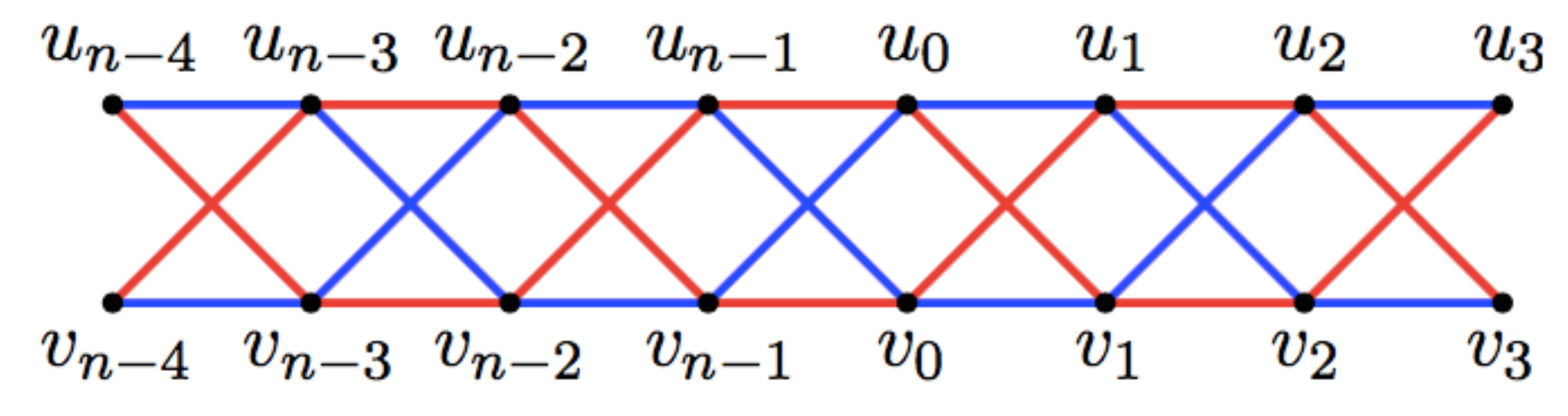}
\caption{The $2n$-faces of the map that has $2n$-faces and $n$-faces.}
\label{fig:2n-and-n-faces}
\end{center}
\end{figure}
Again (see Figure~\ref{fig:2n-and-n-faces}), $\rho, \mu \in \Aut(\m)$.
 Letting $\Phi$ be the flag corresponding to the vertex $u_0$, edge $u_0u_1$ and the corresponding $n$-face, it is clear that $\mu\rho = \alpha_0$, $\mu = \alpha_1$ and $\mu t_1 = \alpha_{212}$, and so Lemma~\ref{le:aux} implies that $\m$ is a map of class~$\2$. If however $f'$ is a $4$-face, then we get a map $\m$ with $\rho, \mu \in \Aut(\m)$ and $\mu\rho = \alpha_0$, $t_1 = \alpha_1$ and $\mu = \alpha_{212}$, where $\Phi$ is the flag corresponding to the vertex $u_0$, edge $u_0u_1$ and the $4$-face $(u_0,u_1,v_0,v_1)$. Thus $\m$ is again a map of class $\2$.

The case when $\m$ has faces of lengths $4$ and $2n$ can be dealt with in a similar way. Letting $f = (u_0,u_1, \ldots, u_{n-1},v_0,v_1, \ldots,v_{n-1})$ be one of the faces of length $2n$ it is easy to see that Lemma~\ref{le:cons} forces $T$ to be as in (\ref{eq:T2}), and so $n$ is even. This time $\rho\sigma_0, \mu t_3\sigma_0 \in \Aut(\m)$, and consequently $\mu t_3 \sigma_0 \rho\sigma_0 = \alpha_0$, $t_1 = \alpha_1$ and $\mu t_3\sigma_0 = \alpha_{212}$, where $\Phi$ is the flag corresponding to the vertex $u_0$, edge $u_0u_1$ and the $4$-face $(u_0,u_1,v_0,v_1)$. We therefore get a map of class $\2$. Details are left to the reader. 

To prove that all the obtained maps are pairwise nonisomorphic observe that this is clearly true if the two maps under consideration have faces of different lengths. As for the maps from items (v) and (vi) of the Theorem note that the maps with faces of lengths $4$ and $n$ corresponding to the case when $T$ is as in (\ref{eq:T1}) are such that a face of length $n$ meets all other three faces of length $n$ while in the case when $T$ is as in (\ref{eq:T2}) this does not hold, proving that the corresponding maps cannot be isomorphic.
\hfill $\Box$
\end{proof}

\subsection{Family (ii)} 

We now consider the second family of arc transitive Rose Window Graphs, namely, the graphs of the form $R_{2n}(n+2,n+1)$. Using a suitable computer package one can verify that the graph $R_6(5,4)$ is the underlying graph of three maps of class~$\2$. Their face lengths are $3$ and $4$, $3$ and $6$ and $4$ and $6$, respectively. Similarly, the graph $R_8(6,5)$ is the underlying graph of exactly one map of class~$\2$. Its face lengths are $4$ and $8$. For the rest of this section we can thus assume $n \geq 5$.

The graph $\G = R_{2n}(n+2,n+1)$ is isomorphic to the Praeger-Xu graph $C(2,n,2)$ (\cite{PraXu89}; see also \cite[Section~5]{Wil08}). For convenience we relabel the vertices of $\G$ by setting:
$$
	u_i = \left\{\begin{array}{lcl}
	x_i & ; & 0 \leq i \leq n-2\\
	y_{n-2} & ; & i = n-1\end{array}\right., \quad 
	v_i = \left\{\begin{array}{lcl}
	y_{i-1} & ; & 1 \leq i \leq n-2\\
	y_{2n-1} & ; & i = 0\\
	x_{n-1} & ; & i = n-1\end{array}\right.,	
$$
$$
	w_i = \left\{\begin{array}{lcl}
	y_{n+i-1} & ; & 0 \leq i \leq n-2\\
	x_{2n-1} & ; & i = n-1\end{array}\right.\mathrm{and} \quad 
	z_i = \left\{\begin{array}{lcl}
	x_{n+i} & ; & 0 \leq i \leq n-2\\
	y_{2n-2} & ; & i = n-1.\end{array}\right.
$$
Note that now the indices for $u_i$, $v_i$, $w_i$ and $z_i$ can be taken in $\ZZ_n$, and we do so.

\begin{figure}[!ht]
\centering
\includegraphics[scale = 0.42]{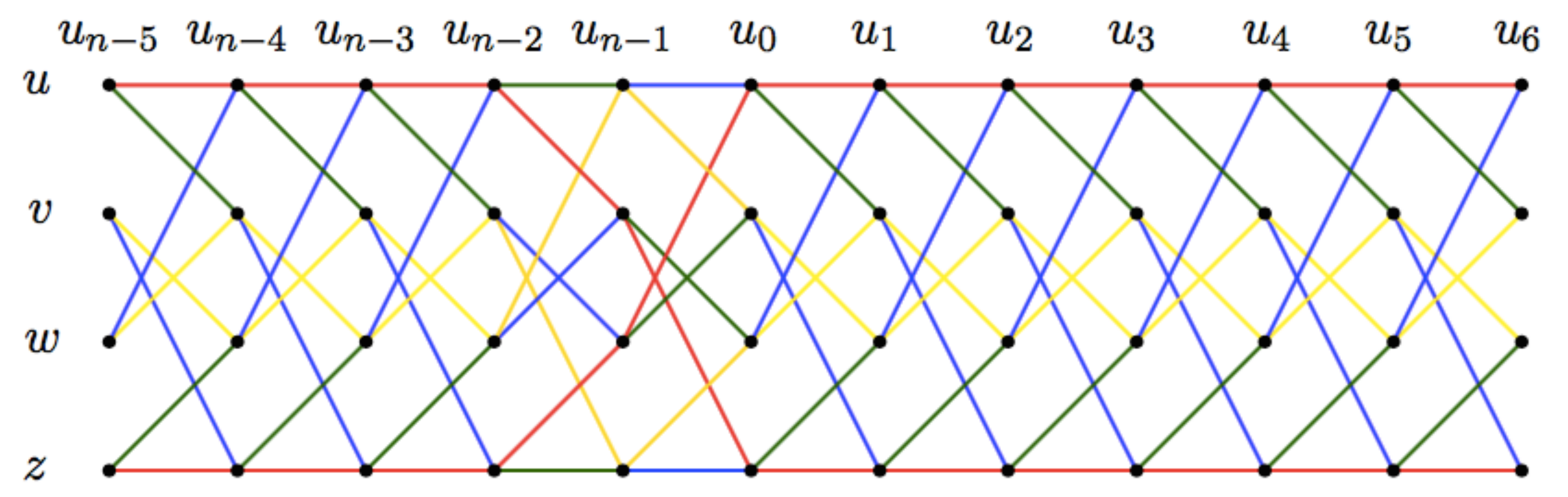}
\caption{Relabeling of the vertices of the graph $R_{2n}(n+2,n+1)$.}\label{fig:relabel}
\end{figure}

A presentation of $\G$ with respect to this relabeling, which we will be relying on in the reminder of this subsection, is given in Figure~\ref{fig:relabel}. The rim edges are colored red, the hub edges yellow, the in-spokes green and the out-spokes blue. 
%

Let us define $\sigma_i = (u_i, v_i)(w_i,z_i)(u_{i+1},w_{i+1})(v_{i+1},z_{i+1})$ for each $i \in \{0,1,\ldots,n-1\}$, where the indices are taken in $\ZZ_{n}$. It is clear that the $\sigma_i$ are automorphisms of $\G$. (The $\sigma_i$ correspond to the permutations $\epsilon_i$ of \cite{KovKutRuf10} and $\sigma_i$ of \cite{Wil08}).
Note that, for each $i$, $\sigma_i\sigma_{i+1}=\sigma_{i+1}\sigma_i$, and so $\sigma_i\sigma_j=\sigma_j\sigma_i$ holds for every $i, j \in  \{0,1,\ldots,n-1\}$. We denote the (elementary abelian) group generated by $\sigma_0, \ldots, \sigma_{n-1}$ by $N$.

Now, set $\alpha =\rho\sigma_{n-1}$ and observe that 
\begin{equation}
\label{eq:alpha}
\alpha=(u_0,u_1,\ldots, u_{n-1})(v_0,v_1,\ldots, v_{n-1})(w_0,w_1,\ldots, w_{n-1})(z_0,z_1,\ldots, z_{n-1}),
\end{equation}
and thus, $\sigma_i=\sigma_{0}^{\alpha^{i}}$.
This shows that the $n$-cycles $(u_0,u_1,\ldots, u_{n-1})$ and $(z_0,z_1,\ldots, z_{n-1})$ are $\Aut(\G)$-consistent cycles of $\G$. Next, let 
\begin{equation}
\label{eq:beta}
\beta = \left(\prod\limits_{i=1}^{\lfloor \frac{n-1}{2}\rfloor} (u_i, u_{n-i})(z_i, z_{n-i})\right) (v_0,w_0) \left(\prod\limits_{i=1}^{n-1} (v_i, w_{n-i})\right).
\end{equation}
We can think of $\beta$ being the ``twisted'' reflection with respect to the ``line through $u_0$ and $z_0$'' in Figure~\ref{fig:relabel}, which interchanges the roles of the $v_i$ and $w_j$ vertices. It is easy to verify that $\eta = \sigma_1\sigma_2\cdots \sigma_{n-2}\beta$, and so $\beta \in \Aut(\G)$. Moreover, $\Aut(\G) = \langle \rho, \mu, \sigma_0\rangle = \langle \alpha, \beta, \sigma_0\rangle$ (see also~\cite{Wil08}).

The graph $\G$ has several blocks of imprimitivity for the action of its automorphism group. For instance, one can verify that the sets $B_i=\{u_i,v_i,w_i,z_i\}$ are blocks for each $i$ (recall that $n > 4$). Consequently, the subsets $\{u_i,z_i\}$ and $\{v_i,w_i\}$ are also blocks for $\Aut(\G)$ (since these are the only $2$-subsets of vertices of $B_i$ which do not lie on a common $4$-cycle). 
\medskip

For the rest of this subsection we assume that $\G$ is the underlying graph of a map $\mathcal{M}$ of type $2_{\{0,1\}}$ and we let $T=N \cap \Aut (\mathcal{M})$.
 Since $\Aut(\mathcal{M})$ is $1$-regular it follows that $|T| = 8$. Moreover, as in the case of Family (i) $T$ is normal in $\Aut(\m)$ and $1$-regularity of $\Aut(\m)$ implies that no nontrivial element of $T$ fixes an arc of $\G$. Observe that since each element of $N$ fixes each $B_i$ setwise any $2$-element subset of $B_i$ is a block of imprimitivity for the restriction of the action of $N$ (and thus $T$) on $B_i$. This proves that any $t \in T$ fixing a vertex of $B_i$ must fix $B_i$ pointwise. Thus, an element $t\in T$ has one of the following actions on the block $B_i$:
 \begin{eqnarray}
 1, \ \ (u_i,v_i)(w_i,z_i), \ \ (u_i,w_i)(v_i,z_i), \ \mathrm{or} \ \ (u_i,z_i)(v_i,w_i).
 \end{eqnarray}
For $t \in T$ we shall write $t=(j_0,j_1,\ldots,j_{n-1})$, where $j_i$ is equal to $0$, $1$, $2$ or $3$, depending on whether the action of $t$ on $B_i$ is trivial, or is $(u_i,v_i)(w_i,z_i)$, $(u_i,w_i)(v_i,z_i)$ or $(u_i,z_i)(v_i,w_i)$, respectively. We can now determine all the possibilities for the subgroup $T$.

\begin{lemma}
\label{le:T(ii)}
We either have 
\begin{equation}
\label{eq:T3}
	\begin{array}{c} T = \{(0,0,\ldots,0), (0,1,2,0,1,2,\ldots,0,1,2),\\
	(2,0,1,2,0,1,\ldots,2,0,1), (1,2,0,1,2,0,\ldots,1,2,0),\\
        (2,1,3,2,1,3,\ldots,2,1,3), (3,2,1,3,2,1,\ldots,3,2,1), \\
	 (1,3,2,1,3,2,\ldots,1,3,2),(3,3,\ldots,3)
	  \},\end{array}
\end{equation}
in which case $3\mid n$, or
\begin{equation}
\label{eq:T4}
	\begin{array}{c} T = \{(0,0,\ldots,0), (0,1,3,2,0,1,3,2,\ldots,0,1,3,2),\\
	(2,0,1,3,2,0,1,3,\ldots,2,0,1,3), (3,2,0,1,3,2,0,1,\ldots,3,2,0,1),\\
        (1,3,2,0,1,3,2,0,\ldots,1,3,2,0), (2,1,2,1,\ldots,2,1), \\
	(1,2,1,2,\ldots,1,2),(3,3,\ldots,3)
	  \},\end{array}
\end{equation}
in which case $4\mid n$. In particular, $\gcd(n,12) \neq 1$.
\end{lemma}

\begin{proof}
As in the proof of Lemma~\ref{le:T(i)} we can show that the fact that $T$ is a normal subgroup of $\Aut(\m)$ and $\Aut(\m)$ is $1$-regular implies that $T^{\alpha}=T$ and $T^{\beta}=T$.
That is, given $t\in T$, both $t^{\alpha}$ and $t^\beta$ are in $T$. Note that if $t=(j_0, j_1, \dots, j_{n-1})$, then $t^\alpha = (j_{n-1}, j_0, \dots, j_{n-1})$, and $t^\beta = (k_0, k_1, k_2, \dots, k_{n-1})$, where $k_i=j_{n-i}$ whenever $j_{n-i} \in \{0, 3\}$, $k_i=2$ if $j_{n-i}=1$ and $k_i=1$ if $j_{n-i}=2$.

Consider the block $B_i$ and let $t \in T$.
Observe that if $j_i=0$, since the common neighbours of $u_i$ and $w_i$ are $u_{i+1}$ and $v_{i+1}$, then $j_{i+1}\in \{0,1\}$. 
Analogously if $j_{i}=2$, then $j_{i+1}\in \{0,1\}$, and if $j_{i}\in \{1,3\}$, then $j_{i+1}\in \{2,3\}$.
We also note that $1$-regularity of $\Aut(\m)$ implies that the identity is the only element of $T$ such that $j_i = j_{i+1}=0$ holds for some $i$. In particular, for a non-identity $t \in T$ every $0$ must be followed by a $1$ and be preceded by a $2$. Similarly, for any $t=(j_0,j_1, \dots, j_{n-1})$ and $t'=(k_0, k_1, \dots, k_{n-1})$ in $T$ if for some $i$, $j_i=k_i$ and $j_{i+1}=k_{i+1}$, then $t = t'$. 

Let $T=\{1,t_1,t_2,t_3,t_4,t_5,t_6,t_7\}$. Without loss of generality let $t_1=(0,1,i_2,i_3,\ldots,i_{n-2}, 2)$ and  $t_2 = t_1^{\alpha}$. Recall that $i_2 \in \{2,3\}$. We consider each of the two cases separately. 
\medskip

\noindent
Case 1: $i_2=2$.\\
Then $t_1=(0,1,2,i_3,\ldots,i_{n-2},2)$ and $t_1^\beta = (0,1,k_2,k_3,\ldots , k_{n-3},1,2)$, and so the above remarks imply that $t_1^\beta = t_1 = (0,1,2,i_3, \ldots , i_{n-3},1,2)$. Then $t_1^{\alpha^3} = (i_{n-3},1,2,0,1,2,i_3,\ldots, i_{n-4})$, and so $t_1^{\alpha^3} = t_1$, implying that $n$ is divisible by $3$ and $t_1 = (0,1,2,0,1,2,\ldots , 0,1,2)$. The subgroup $T$ is now completely determined: 
$$ 
\begin{array}{ccccl} 
 &  & t_1 & =& (0,1,2,0,1,2,\ldots,0,1,2), \\    
t_1^{\alpha}  & = & t_2 & = & (2,0,1,2,0,1,\ldots,2,0,1), \\    
t_2^{\alpha} &  =& t_3 & = & (1,2,0,1,2,0,\ldots,1,2,0), \\    
t_1t_2 & = & t_4 & = & (2,1,3,2,1,3,\ldots,2,1,3), \\    
t_4^{\alpha} &= & t_5 & = & (3,2,1,3,2,1,\ldots,3,2,1), \\       
t_5^{\alpha} & = & t_6 & = &(1,3,2,1,3,2,\ldots,1,3,2), \\    
t_3t_4&  =& t_7& = & (3,3,3,3,3,3,\ldots,3,3,3).
\end{array}
$$

\noindent
Case 2: $i_2=3$.\\
In this case we have $t_1=(0,1,3,i_3,i_4,\ldots, i_{n-2},2)$. As before $t_1^{\beta} = t_1$, and so $t_1 = (0,1,3,i_3,\ldots , i_{n-3},3,2)$. Set $t_2 = t_1^\alpha = (2,0,1,3,i_3,\ldots , i_{n-3},3)$ and $t_3 = t_1^{\alpha^2} = (3,2,0,1,3,i_3,\ldots, i_{n-3})$. Then $t = t_1t_2t_3 = (1,3,2,j_3,j_4,\ldots , j_{n-1})$. Since $t_4 = t_1^{\alpha^3} = (i_{n-3},3,2,0,1,3,i_3,\ldots , i_{n-4})$ it follows that $i_{n-3} = 1$, and so $t_1^{\alpha^4} = t_1$, implying that $n$ is divisible by $4$ and $t_1 = (0,1,3,2,0,1,3,2,\ldots , 0,1,3,2)$. The subgroup $T$ is now completely determined:
\begin{equation*}
 \begin{tabular}{ccccl} 
 &  & $t_1$ & $=$ & $(0,1,3,2,0,1,3,2,\ldots,0,1,3,2)$, \\    
$t_1^{\alpha}$  & = & $t_2$ & $=$ & $(2,0,1,3,2,0,1,3,\ldots,2,0,1,3)$, \\    
$t_2^{\alpha}$ &  =& $t_3$ & $=$ & $(3,2,0,1,3,2,0,1,\ldots,3,2,0,1)$, \\    
$t_3^{\alpha}$ &  =& $t_4$ & $=$ & $(1,3,2,0,1,3,2,0,\ldots,1,3,2,0)$, \\    
$t_1t_2$ & = & $t_5$ & $=$ & $(2,1,2,1,2,1,2,1,\ldots,2,1,2,1)$, \\       
 $t_5^{\alpha}$ & = & $t_6$ & $=$ & $(1,2,1,2,1,2,1,2,\ldots,1,2,1,2)$, \\    
 $t_1t_3$&  =& $t_7$ & $=$ & $(3,3,3,3,3,3,3,3,\ldots,3,3,3,3)$.
\end{tabular}
\end{equation*}
\hfill $\Box$
\end{proof}
\bigskip

Recall that, by Lemma~\ref{facesofhereditarymaps}, the boundaries of faces of a map $\m$ of class $\2$ are $\Aut(\m)$-consistent cycles. By Proposition~\ref{pro:cons} the graph $\G$ has three orbits of $\Aut(\G)$-consistent cycles. Since $\sigma_1\beta\alpha, \alpha, \rho \in \Aut(\G)$ are shunts for the cycles 
$$ (u_0,u_1,w_0,v_1),\ (u_0,u_1,u_2, \ldots,u_{n-1}) \ \mathrm{and}$$
$$ (u_0,u_1,\ldots,u_{n-3}, u_{n-2},v_{n-1},z_0,z_1,\ldots,z_{n-3}, z_{n-2},w_{n-1}),$$
respectively, these three cycles are representatives of the three orbits of $\Aut(\G)$-consistent cycles. Therefore, the $\Aut(\G)$-consistent cycles are of lengths $4$, $n$ and $2n$, implying that these are the only possible lengths of faces of $\m$. We remark that $n \neq 4$ implies that each edge of $\G$ is in exactly one $4$-cycle. The proof of the following lemma is similar to that of Lemma~\ref{le:lengths} and is left to the reader (it again relies on the fact that any two $\Aut(\G)$-consistent cycles of length $n$ or $2n$ are permutable by an element of the subgroup $N$).
 
 \begin{lemma}
\label{le:lengths2}
The map $\m$ has faces of two different lengths.
\end{lemma}

We can now describe all maps $\m$ of type $\2$ whose underlying graph is $\G$. 

\begin{theorem}
\label{th:RW(ii)}
Let $\G = R_{2n}(n+2,n+1)$ be a Rose Window graph with $n \geq 3$. Then $\G$ is the underlying graph of a map $\m$ of class $\2$ if and only if $\gcd(n,12) > 2$. Moreover, letting $n_0 \in \{0,3,4,6,8,9\}$ be the residue of $n$ modulo $12$, the following holds:
\begin{itemize}\itemsep = 0pt
\item[(i)] if $n = 4$, then $\G$ is the underlying graph of a unique map of class~$\2$ with face lengths $4$ and $8$.
\item[(ii)] if $n_0 \in \{3, 9\}$, then $\G$ is the underlying graph of three nonisomorphic maps of class~$\2$; one has faces of lengths $4$ and $n$, one has faces of lengths $4$ and $2n$, and one has faces of lengths $n$ and $2n$.
\item[(iii)] if $n_0 \in \{4, 6, 8\}$, then $\G$ is the underlying graph of two nonisomorphic maps of class~$\2$; one has faces of lengths $4$ and $n$, while the other has faces of lengths $4$ and $2n$.
\item[(iv)] if $n_0 = 0$, then $\G$ is the underlying graph of four nonisomorphic maps of class~$\2$; two have faces of lengths $4$ and $n$, and two has faces of lengths $4$ and $2n$.
\end{itemize}
\end{theorem}
  
\begin{proof}
The cases $n = 3$ and $n = 4$ have been dealt with at the beginning of this subsection. For the rest of the proof we can thus assume $n > 4$. We distinguish the cases depending on whether $\m$ has $n$-faces or not. 
\medskip

\noindent
Case 1: $\m$ has an orbit of $n$-faces.\\
Without loss of generality we can assume that one of the $n$-faces of $\m$ is $f=(u_0,u_1,\ldots,u_{n-1})$. As in the proof of Theorem~\ref{th:RW(i)}, we deal with the two possibilities for $T$ separately.
\smallskip
 
\noindent
Subcase 1.1: $T$ is as in (\ref{eq:T3}), in which case $3$ divides $n$.\\
By Lemma~\ref{le:T(ii)} the eight $n$-faces of $\m$ are as represented in Figure~\ref{fig:family2-n4T3}, that is:
\begin{equation}
\label{eq:nfacesT5}
 \begin{tabular}{rcccl} 
$f_0$ & $=$ & $f$ & $=$ & $(u_0,u_1,\ldots,u_{n-1})$,\\    
$f_1$ & $=$ &$ft_1$ & $=$ & $(u_0,v_1,w_2,u_3,v_4,w_5, \ldots,u_{n-3},v_{n-2},w_{n-1})$, \\    
$f_2$ & $=$ &$ft_2$ & $=$ & $(w_0,u_1,v_2,w_3,u_4,v_5, \ldots,w_{n-3},u_{n-2},v_{n-1})$, \\    
$f_3$ & $=$ &$ft_3$ & $=$ & $(v_0,w_1,u_2,v_3,w_4,u_5, \ldots,v_{n-3},w_{n-2},u_{n-1})$, \\    
$f_4$ & $=$ &$ft_4$ & $=$ & $(w_0,v_1,z_2,w_3,v_4,z_5, \ldots,w_{n-3},v_{n-2},z_{n-1})$, \\
$f_5$ & $=$ &$ft_5$ & $=$ & $(z_0,w_1,v_2,z_3,w_4,v_5, \ldots,z_{n-3},w_{n-2},v_{n-1})$, \\
$f_6$ & $=$ &$ft_6$ & $=$ & $(v_0,z_1,w_2,v_3,z_4,w_5, \ldots,v_{n-3},z_{n-2},w_{n-1})$, \\
$f_7$ & $=$ &$ft_7$ & $=$ & $(z_0,z_1,\ldots,z_{n-3},z_{n-1})$. \\    
\end{tabular}
\end{equation}

\begin{figure}[!ht]
\centering
\includegraphics[scale = 0.42]{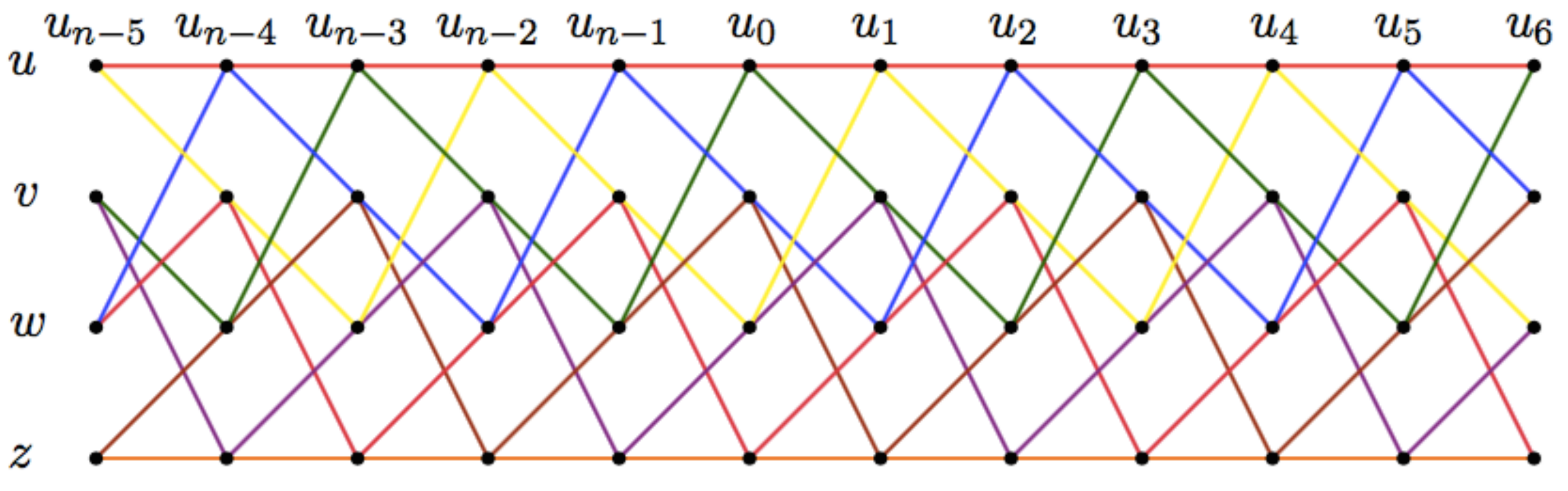}
\caption{The $n$-faces of $\m$ in the case when $T$ is as in (\ref{eq:T3}).}\label{fig:family2-n4T3}
\end{figure}

The automorphism $\beta \in \Aut(\G)$ from (\ref{eq:beta}) preserves the set of $n$-cycles (\ref{eq:nfacesT5}) since it fixes each of the faces $f_0, f_1, f_5$ and $f_7$, and interchanges $f_2$ with $f_3$ and $f_4$ with $f_6$. Since the other $\Aut(\m)$-orbit of faces of $\mathcal{M}$ either consists of $4$-faces or $2n$-faces and each edge of $\G$ lies on a unique $4$-cycle (recall that $n > 4$), while Lemma~\ref{facesofhereditarymaps} and Corollary~\ref{cor:twoedges} imply that the $2n$-faces are uniquely determined by the $n$-faces, the corresponding maps are completely determined (once we have decided for the length of the faces of the other $\Aut(\m)$-orbit). Moreover, since the automorphism $\beta$ preserves the set of the eight $n$-faces, it in fact follows that $\beta \in \Aut(\m)$. We thus only need to check if the resulting maps are indeed of class~$\2$. 

Suppose first that $\m$ has $4$-faces (and the $n$-faces from (\ref{eq:nfacesT5})). We show that $\m$ is in class $\2$ by exhibiting the automorphisms $\alpha_0$, $\alpha_1$ and $\alpha_{212}$, with respect to some base flag, so that we can apply Lemma~\ref{le:aux}. Let $\Phi$ be the flag of $\m$ containing the vertex $u_{0}$, the edge $u_0u_{1}$ and the $n$-face $f_0$. Then $\Phi \beta\alpha = \Phi^0$, $\Phi\beta = \Phi^1$ and $\Phi t_1 = \Phi^{2,1,2}$, and so $\m$ is of class~$\2$ by Lemma~\ref{le:aux}. 

Suppose now that $\m$ has $2n$-faces (and the $n$-faces from (\ref{eq:nfacesT5})). By the above remarks the $2n$-faces are completely determined. In fact, $n$ has to be odd for this to be possible (that is $n \equiv 3 \pmod{6}$) and in this case the four $2n$-faces are:
{\small \begin{equation*}
 \begin{tabular}{l} 
 $(u_0, u_1, v_2, z_3, z_4, w_5, \ldots, u_{n-3}, u_{n-2}, v_{n-1}, z_0, z_1, w_2, \ldots , u_{n-6}, u_{n-5}, v_{n-4}, z_{n-3}, z_{n-2}, w_{n-1})$,\\    
$(u_0, v_1, z_2, z_3, w_4, u_5, \ldots, u_{n-3}, v_{n-2}, z_{n-1}, z_0, w_1, u_2, \ldots , u_{n-6}, v_{n-5}, z_{n-4}, z_{n-3}, w_{n-2}, u_{n-1})$,\\ 
$(v_0, z_1, z_2, w_3, u_4, u_5,  \ldots, v_{n-3}, z_{n-2}, z_{n-1}, w_0, u_1, u_2,  \ldots , v_{n-6}, z_{n-5}, z_{n-4}, w_{n-3}, u_{n-2}, u_{n-1})$,\\ 
$(v_0, w_1, v_2, w_3, \ldots, v_{n-1}, w_0, v_1, w_2, \ldots , v_{n-2}, w_{n-1})$.  
\end{tabular}
\end{equation*}
}
Again let $\Phi$ be the flag corresponding to the vertex $u_0$, edge $u_0 u_{1}$ and the face $f_0$. As before we get $\Phi \beta\alpha = \Phi^0$, $\Phi\beta = \Phi^1$, while this time $\Phi t_1\beta = \Phi^{2,1,2}$, and so we can again apply Lemma~\ref{le:aux} to show that $\m$ is of class~$\2$.
\smallskip

\noindent
Subcase 1.2: $T$ is as in (\ref{eq:T4}), in which case $4$ divides $n$.\\
By Lemma~\ref{le:T(ii)} the eight $n$-faces of $\m$ are as represented in Figure~\ref{fig:family2-nT4}, that is:
\begin{equation}
\label{eq:T6}
 \begin{tabular}{rcccl} 
$f_0$ & $=$ & $f$ & $=$ & $(u_0,u_1,\ldots , u_{n-1})$,\\    
$f_1$ & $=$ &$ft_1$ & $=$ & $(u_0,v_1,z_2,w_3,u_4,v_5,z_6,w_7,\ldots,u_{n-4}, v_{n-3},z_{n-2},w_{n-1})$,\\       
$f_2$ & $=$ &$ft_2$ & $=$ & $(w_0,u_1,v_2,z_3,w_4,u_5,v_6,z_7,\ldots,w_{n-4}, u_{n-3},v_{n-2},z_{n-1})$,\\     
$f_3$ & $=$ &$ft_3$ & $=$ & $(z_0,w_1,u_2,v_3,z_4,w_5,u_6,v_7,\ldots,z_{n-4}, w_{n-3},u_{n-2},v_{n-1})$,\\    
$f_4$ & $=$ &$ft_4$ & $=$ & $(v_0,z_1,w_2,u_3,v_4,z_5,w_6,u_7,\ldots,v_{n-4}, z_{n-3},w_{n-2},u_{n-1})$,\\ 
$f_5$ & $=$ &$ft_5$ & $=$ &$(w_0,v_1,w_2,v_3,\ldots,w_{n-2},v_{n-1})$,\\ 
$f_6$ & $=$ &$ft_6$ & $=$ & $(v_0,w_1,v_2,w_3,\ldots,v_{n-2},w_{n-1})$,\\ 
$f_7$ & $=$ &$ft_7$ & $=$ & $(z_0,z_1,\ldots,z_{n-1})$.\\    
\end{tabular}
\end{equation}
\begin{figure}[!ht]
\centering
\includegraphics[scale = 0.3]{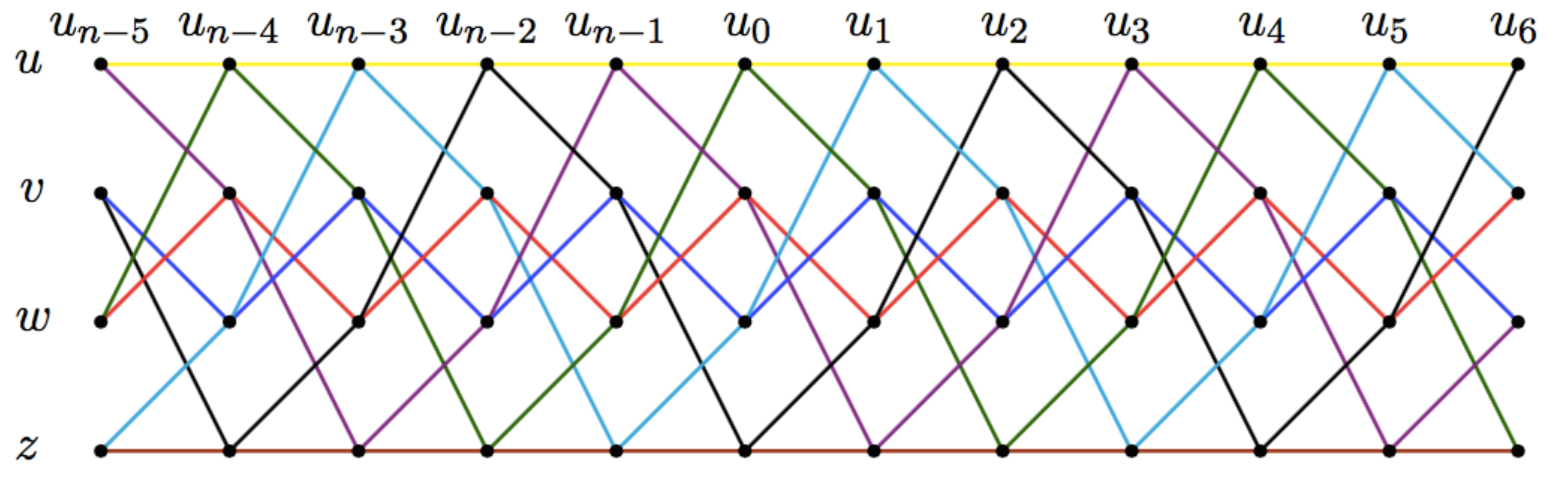}
\caption{The $n$-faces of $\m$ in the case when $T$ is as in (\ref{eq:T4}).}\label{fig:family2-nT4}
\end{figure} 

Again, $\beta$ from (\ref{eq:beta}) preserves the set of eight $n$-cycles (\ref{eq:T6}), implying that $\beta \in \Aut(\m)$. We show that the other faces of $\m$ must be of length $4$. Namely, if this was not the case then, by Corollary~\ref{cor:twoedges}, the other face containing the edge $u_0u_1$ would have to be 
$$
(u_0, u_1, v_2, w_3, u_4, u_5, v_6, w_7, \ldots, u_{n-4}, u_{n-3}, v_{n-2}, w_{n-1}),
$$
which is of length $n$, contradicting Lemma~\ref{le:lengths2}. Thus $\m$ contains $4$-faces (and the $n$-faces from (\ref{eq:T6})). Let $\Phi$ be the flag of $\m$ containing the vertex $u_{0}$, the edge $u_0u_{1}$ and the $n$-face $f_0$, and observe that $\Phi\beta = \Phi^1$, $\Phi \beta\alpha = \Phi^0$ and $\Phi t_1=\Phi^{2,1,2}$. Thus Lemma~\ref{le:aux} implies that $\m$ is in class $\2$. 
\medskip

\noindent
Case 2: $\m$ has faces of lengths $4$ and $2n$.\\
Recall that any automorphism of $\G$ preserves the set of $4$-cycles of $\G$, and so an automorphism of $\G$ is an automorphism of $\m$ if and only if it preserves the set of $2n$-faces. Without loss of generality we can assume that 
$$
f=(u_0,u_1,\ldots,u_{n-2},v_{n-1},z_0,z_1,\ldots,z_{n-2},w_{n-1})
$$ 
is one of the faces of $\m$. Note that $f\mu = f$ and $f\rho = f$ (in fact, $\rho$ is a shunt for $f$). Since $\mu$ and $\rho$ both normalize the subgroup $T$ from Lemma~\ref{le:T(ii)}, it follows that $\mu, \rho \in \Aut(\m)$. We again distinguish the two possibilities for the subgroup $T$. \smallskip

\noindent
Subcase 2.1: $T$ is as in (\ref{eq:T3}), in which case $3$ divides $n$.\\
The $2n$-faces are then (see Figure~\ref{fig:family2-2nT3}):
 \begin{equation*}
\label{eq:T3-2n-2}
 \begin{tabular}{ccl} 
$f$ & $=$ & $(u_0, u_1, u_2,\ldots, u_{n-2},v_{n-1}, z_0,z_1,\ldots, z_{n-2},w_{n-1})$,\\    
$ft_1$ & $=$ & $(u_0, v_1, w_2,u_3, v_4, w_5, \ldots, v_{n-2}, z_{n-1}, z_{0},w_1,v_2,z_3,\ldots, w_{n-2}, u_{n-1})$, \\ 
$ft_2$ & $=$ & $(w_0,u_1,v_2,w_3, u_4, v_5, \ldots u_{n-2}, u_{n-1},v_0,z_1,w_2,v_3,z_4 \dots, z_{n-2}, z_{n-1})$, \\    
$ft_3$ & $=$ & $(v_0, w_1, u_2, v_3, w_4,u_5, \dots, w_{n-2}, v_{n-1}, w_0, v_1, z_2, w_3,v_4,\ldots, v_{n-2}, w_{n-1})$.  
\end{tabular}
\end{equation*}

 \begin{figure}[!ht]
\centering
\includegraphics[scale = 0.4]{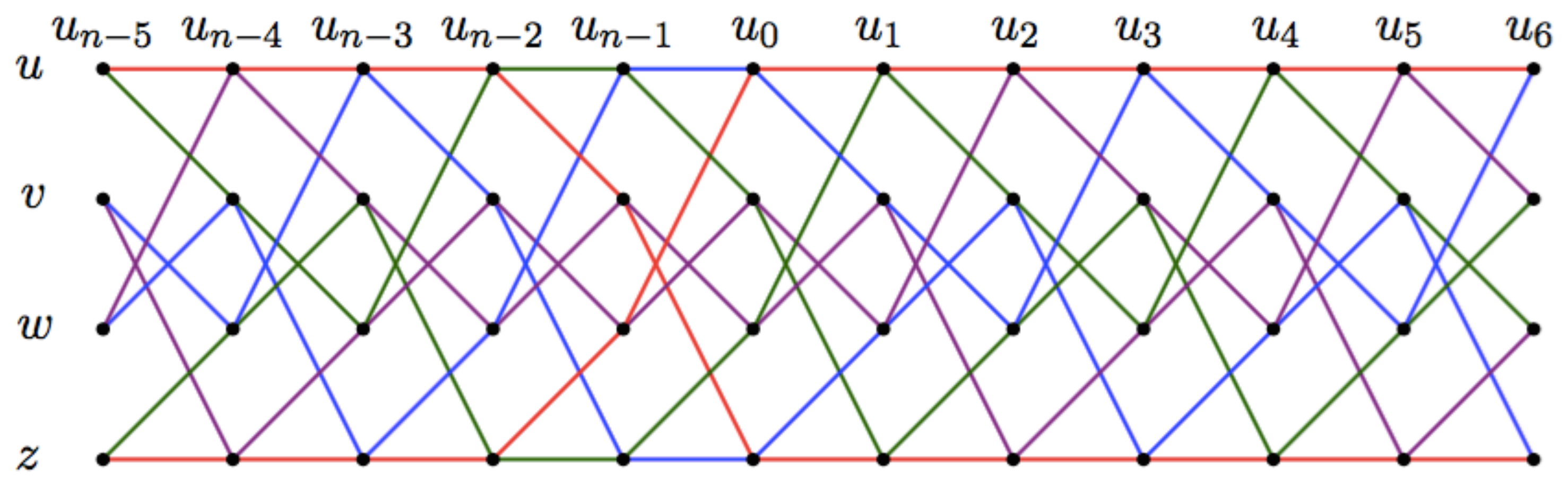}
\caption{The $2n$-faces of $\m$ in the case when $T$ is as in (\ref{eq:T3}).}\label{fig:family2-2nT3}
\end{figure} 

Let $\Phi$ be the flag of $\m$ containing the vertex $u_{0}$, the edge $u_0w_{n-1}$ and the $2n$-face $f$. Then $\Phi \rho\mu = \Phi^0$, $\Phi\mu = \Phi^1$ and $\Phi t_1 = \Phi^{2,1,2}$, and so Lemma~\ref{le:aux} implies that $\m$ is of class~$\2$.
\smallskip

\noindent
Subcase 2.2: $T$ is as in (\ref{eq:T4}), in which case $4$ divides $n$.\\
The $2n$-faces are then (see Figure~\ref{fig:family2-2nT4}):
 \begin{equation*}
\label{eq:T4-2n-2}
 \begin{tabular}{ccl} 
$f$ & $=$ & $(u_0, u_1, u_2,\ldots, u_{n-2},v_{n-1}, z_0,z_1,\ldots, z_{n-2},w_{n-1})$,\\    
$ft_1$ & $=$ & $(u_0, v_1, z_2,w_3, u_4, v_5, \ldots, z_{n-2}, z_{n-1}, z_{0},w_1,u_2,v_3,\ldots, u_{n-2}, u_{n-1})$, \\ 
$ft_2$ & $=$ & $(w_0,u_1,v_2,z_3, w_4, u_5, \ldots v_{n-2}, w_{n-1},v_0,z_1,w_2,u_3,v_4 \dots, w_{n-2}, v_{n-1})$, \\    
$ft_3$ & $=$ & $(v_0, w_1, v_2, w_3, v_4,w_5, \dots, v_{n-2}, z_{n-1}, w_0, v_1, w_2, v_3,w_4,\ldots, w_{n-2}, u_{n-1})$.  
\end{tabular}
\end{equation*}

 \begin{figure}[!ht]
\centering
\includegraphics[scale = 0.4]{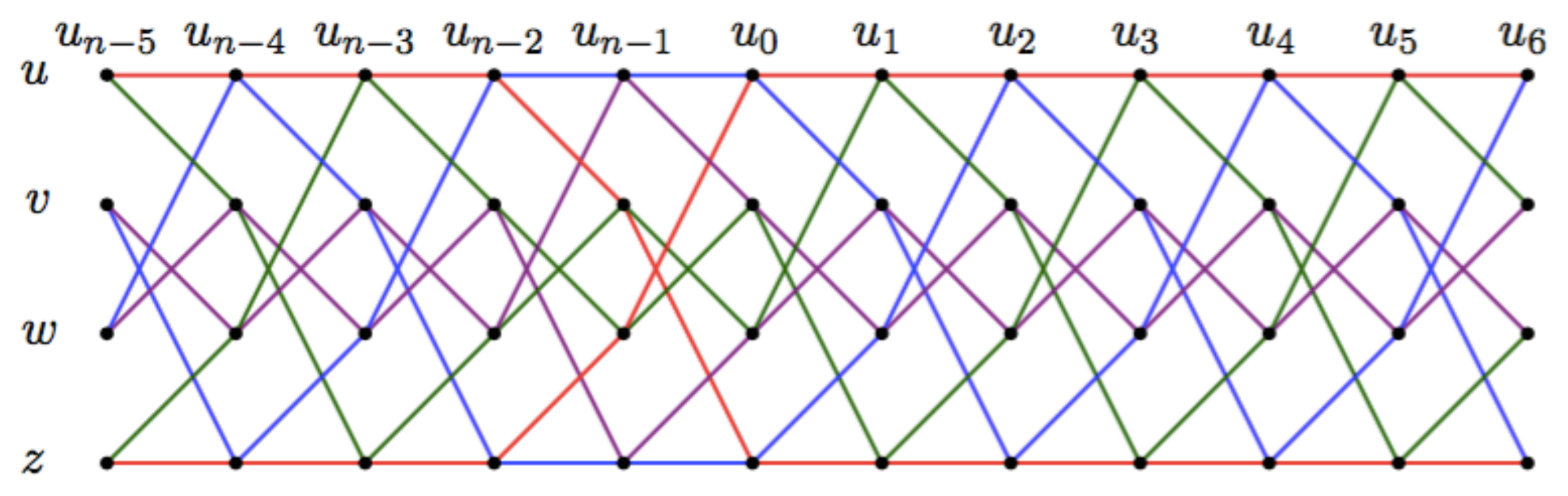}
\caption{The $2n$-faces of $\m$ in the case when $T$ is as in (\ref{eq:T4}).}\label{fig:family2-2nT4}
\end{figure} 
Letting $\Phi$ be as above we again find that $\Phi \rho\mu = \Phi^0$, $\Phi \mu = \Phi^1$ and $\Phi t_1 = \Phi^{2,1,2}$ (note however, that the $t_1$ now differs from the one in the previous paragraph). Thus Lemma~\ref{le:aux} implies that $\m$ is a map of class~$\2$.
\medskip

The proof that all of the obtained maps are pairwise nonisomorphic is similar to the one in the proof of Theorem~\ref{th:RW(i)}.  
\hfill $\Box$
\end{proof}

\subsection{Family (iii)} 

Combining together the results of~\cite{KovKutRuf10} and \cite{Wil08} with Corollary~\ref{cor:1-regular} the classification of maps of class~$\2$ whose underlying graphs belong to family (iii) from Proposition~\ref{pro:RW} is straightforward. 

\begin{theorem}
\label{th:RW(iii)}
Let $\G = R_{2m}(2b,r)$, where $b^2 \equiv \pm 1 \pmod m$ and either $r = 1$, or $r = m-1$ with $m$ even, be such that $\G$ does not belong to any of the families (i) and (ii) from Proposition~\ref{pro:RW}. Then $\G$ is the underlying graph of a map of class~$\2$ if and only if $b^2 \equiv 1 \pmod{m}$ in which case there are exactly three pairwise nonisomorphic such maps. 
\end{theorem}

\begin{proof}
By \cite[Proposition~3.8]{KovKutRuf10} the automorphism group of $\G$ is $1$-regular and is generated by $\rho$, $\mu$ and $\sigma$, where $\rho$ and $\mu$ are as in (\ref{eq:rhomu}) and $\sigma$ is as in \cite[page 16]{Wil08}. It was pointed out in \cite{Wil08} that in the case of $b^2 \equiv -1 \pmod{m}$ the vertex stabilizers are cyclic (and thus isomorphic to $\ZZ_4$) while in the case of $b^2 \equiv 1 \pmod{m}$ they are isomorphic to the Klein $4$-group. We can thus apply Corollary~\ref{cor:1-regular}.
\hfill $\Box$
\end{proof}

\subsection{Family (iv)} 

The fourth family of edge-transitive Rose Window graphs consists of the graphs $R_{12m}(3m+2,3m-1)$ and $R_{12m}(3m-2,3m+1)$, where $m \geq 1$. As in \cite{Wil08}, using the fact that $R_n(a,r) \cong R_n(-a,r) \cong R_n(a,-r)$, we can denote these graphs as $R_{12m}(3d+2, 9d+1)$, where $d=m$ or $d=-m$ (modulo $12m$). In~\cite{Wil08} the following automorphism $\sigma$ of $\G = R_{12m}(3d+2,9d+1)$ has been identified (recall that $a=3d+2$):
$$x_i^{\sigma}=\left\{ \begin{array}{lcl}
x_i&  ; & i \equiv 0 \pmod{3}\\
y_{i-1}&  ; & i \equiv 1 \pmod{3}\\
y_{i+1-a}&  ; & i \equiv 2 \pmod{3}
\end{array}
 \right. \mbox{ and }y_i^{\sigma}=\left\{ \begin{array}{lcl}
x_{i+1}&  ; & i \equiv 0 \pmod{3}\\
x_{i-1+a}&  ; & i \equiv 1 \pmod{3}\\
y_{i+6d}&  ; & i \equiv 2 \pmod{3}.
\end{array}
 \right.$$
 
Moreover, it was shown that whenever $m \equiv 2 \pmod{4}$, setting $b = d+1$, an additional automorphism $\tau$ of $\G$ exists:
 $$x_i^{\tau}=\left\{ \begin{array}{lcl}
x_{bi}&  ; & i \equiv 0 \pmod{3}\\
y_{bi-b}&  ; & i \equiv 1 \pmod{3}\\
x_{bi+b-1}&  ; & i \equiv 2 \pmod{3}
\end{array}
 \right.\mbox{ and }y_i^{\tau}=\left\{ \begin{array}{lcl}
x_{bi+1}&  ; & i \equiv 0 \pmod{3}\\
y_{4+bi-4b}&  ; & i \equiv 1 \pmod{3}\\
y_{bi+b-1}&  ; & i \equiv 2 \pmod{3}.
\end{array}
 \right.$$ 
Observe that $a=3b-1$, $r=4-3b$ and $3b^2\equiv 3 \pmod{12m}$, and so $a \equiv 2 \pmod{3}$ and $r \equiv 1 \pmod{3}$. It was shown in \cite{KovKutRuf10} that $\Aut(\G)=\langle \rho, \mu, \sigma, \tau \rangle$, whenever $m \equiv 2 \pmod{4}$, and $\Aut(\G)=\langle \rho, \mu,\sigma\rangle$ otherwise, where $\rho$ and $\mu$ are as in (\ref{eq:rhomu}). This enables us to classify the maps of class $\2$ with underlying graphs from family (iv).
 
\begin{theorem}
\label{th:RW(iv)}
Let $\G = R_{12m}(3d+2, 9d+1)$ be a Rose Window graph, where $d = m$ or $d = 11m$. Then 
\begin{enumerate} 
\item[(i)] if $m \not\equiv 2 \pmod{4}$, $\G$ is the underlying graph of exactly three nonisomorphic maps of class $\2$, 
\item[(ii)] if $m \equiv 2 \pmod{4}$, $\G$ is the underlying graph of exactly two pairwise nonisomorphic maps of class $\2$.
\end{enumerate}
\end{theorem}

\begin{proof} 
We first deal with the case when $m \not\equiv 2 \pmod{4}$. By \cite[Proposition~3.5]{KovKutRuf10} the automorphism group $\Aut(\G)$ is $1$-regular in this case and is isomorphic to $\langle \rho, \mu, \sigma\rangle$. Note that $\Aut(\G)_{x_0} = \langle \sigma, \mu \rangle \cong \ZZ_2 \times \ZZ_2$, and so the vertex stabilizers in $\Aut(\G)$ are isomorphic to the Klein $4$-group. We can thus apply Corollary~\ref{cor:1-regular} to prove that $\G$ is the underlying graph of three pairwise nonisomorphic maps of class~$\2$.

For the rest of the proof we can thus assume that $m \equiv 2 \pmod{4}$. Recall that in this case $\Aut(\G) = \langle \rho, \mu, \sigma, \tau \rangle$ is arc-transitive with vertex-stabilizers of order $8$. In particular, $\Aut(\G)_{x_0} = \langle \mu, \sigma, \tau \rangle \cong D_4$, the dihedral group of order $8$. It is easy to see that $\sigma$ commutes with both $\mu$ and $\tau$, while $\tau\mu\tau = \mu\sigma$ and $\tau\rho\tau = \rho\sigma$. 

Suppose $\m $ is a map in class $\2$ with underlying graph $\G$ and recall that the automorphism group $\Aut(\m)$ is then $1$-regular on $\G$. By Theorem~\ref{the:1-regular} the boundaries of the faces of $\m$ are $\Aut(\m)$-symmetric consistent cycles, and so Lemma~\ref{le:tetra_cons} implies that the vertex stabilizers in $\Aut(\m)$ are isomorphic to the Klein $4$-group (which of course must be transitive on the neighbourhood of the fixed vertex). It is easy to see that the only subgroup of $\Aut(\G)_{x_0} = \langle \mu, \sigma, \tau\rangle$, transitive on the set of four neighbours of $x_0$ and isomorphic to the Klein $4$-group is $\langle \mu, \sigma\rangle$. Thus, $\Aut(\m)$ is a transitive index $2$ subgroup of $\Aut(\G)$ containing the subgroup $\langle \mu, \sigma\rangle$. 

We claim that $H_1=\langle \mu, \sigma, \rho \rangle$ and $H_2=\langle \mu, \sigma, \tau\rho \rangle$ are the only two such subgroups.
Since $\Aut(\m)$ must be vertex-transitive, it has to contain an element $\gamma \in \Aut(\G)$ mapping $x_0$ to $x_1$. But $\Aut(\G)$ is vertex-transitive with $\Aut(\G)_{x_0} = \langle \mu, \sigma, \tau\rangle$, and so the fact that $\tau$ normalizes $\langle \mu, \sigma \rangle$ implies that $\gamma \in \langle \mu, \sigma \rangle \rho \cup \langle \mu, \sigma\rangle \tau\rho$. It follows that $\Aut(\m)$ could only be one of $H_1$ and $H_2$. We next prove that $H_1$ and $H_2$ are indeed of index $2$ in $\Aut(\G)$ (since they contain $\langle \mu, \sigma \rangle$ and an element mapping $x_0$ to $x_1$, they are both arc-transitive on $\G$). Note that the fact that each $H_i$ acts arc-transitively implies that $H_i$ is of index $2$ in $\Aut(\G)$ or $H_i \cong \Aut(\G)$. It thus suffices to find an element of $\Aut(\G)$ which is not contained in $H_i$. 

We first deal with $H_1$. 
To this end we identify four cycles of $\G$, each of length $12m$. Observe that $a+1 = 3d+3$, implying that $\gcd(12m,a+1) = 3$ (recall that $d = m$ or $d = -m$ and $m$ is even). Therefore, the cycle
$$
C_1 = (x_0,x_1,y_1,x_{a+1},x_{a+2},y_{a+2},x_{2(a+1)},\ldots , y_{-a})
$$
is indeed of length $12m$. Observe that $C_1$ is an $H_1$-consistent cycle with a shunt $\sigma\rho$. Let $C_2 = C_1\rho$ and $C_3 = C_1\rho^2$ be the images of $C_1$ under $\rho$ and $\rho^2$, respectively, and note that $C_2$ and $C_3$ are thus also $H_1$-consistent cycles of $\G$. Finally, observe that $\gcd(12m,r) = 1$ and let $C_4 = (y_0,y_r,y_{2r},\ldots , y_{-r})$. Of course, $C_4$ also is an $H_1$-consistent cycle with a shunt $\rho^r$. It is easy to see that each of the generators $\rho$, $\sigma$ and $\mu$ of $H_1$ preserves the set $\{C_1, C_2, C_3, C_4\}$ setwise ($\rho$ fixes $C_4$ and permutes $C_1$, $C_2$ and $C_3$, $\sigma$ interchanges $C_1$ with $C_3$ and $C_2$ with $C_4$, while $\mu$ interchanges $C_1$ with $C_3$ and fixes both $C_2$ and $C_4$). Since $\tau$ clearly does not map $C_1$ to any of the cycles $C_i$, this implies that $\tau \notin H_1$, and so $H_1$ is indeed of index $2$ in $\Aut(\G)$.

The proof that $H_2$ is also of index $2$ in $\Aut(\G)$ is similar. First we observe that $\gcd(12m, 2+2r+a) = \gcd(12m,9d+6) = 6$, and so the cycle
$$
C_1 = (x_0,x_1,x_2,y_2,y_{2+r},y_{2+2r},x_{2+2r+a},x_{3+2r+a},x_{4+2r+a},\ldots , y_{-a})
$$
is of length $12m$ and is an $H_2$-consistent cycle with a shunt $\sigma\tau\rho$. Setting $C_2 = C_1\mu$, $C_3 = C_1\tau\rho$ and $C_4 = C_1\tau\rho\sigma$ we get three other $H_2$-consistent cycles. 
We can then verify that each of the generators $\tau\rho$, $\sigma$ and $\mu$ of $H_2$ preserve the set $\{C_1, C_2, C_3, C_4\}$ setwise 
Since $\rho$ clearly does not map $C_1$ to any of the cycles $C_i$, this implies that $\rho \notin H_2$, and so $H_2$ is indeed of index $2$ in $\Aut(\G)$.

We now classify the maps $\m$ of class~$\2$ with underlying graph $\G$ and automorphism group isomorphic to $H_1$ or $H_2$ separately. We first deal with the maps $\m$ with $\Aut(\m) = H_1$. By Lemma~\ref{facesofhereditarymaps} the boundaries of the faces of $\m$ are $H_1$-symmetric consistent cycles. Since $H_1$ is $1$-regular, there are exactly three orbits, say $\mathcal{O}_1, \mathcal{O}_2$ and $\mathcal{O}_3$, of $H_1$-consistent directed cycles of $\G$ and by Lemma~\ref{le:cons} there is exactly one $H_1$-consistent directed cycle containing the arc $(x_{-1},x_0)$ from each of these three orbits. Let $\vec{C}_1$, $\vec{C}_2$ and $\vec{C}_3$ be the corresponding $H_1$-consistent directed cycles with $\vec{C}_i \in \mathcal{O}_i$. With no loss of generality we can assume $\vec{C}_1$ contains the $2$-arc $(x_{-1},x_0,x_1)$, $\vec{C}_2$ contains the $2$-arc $(x_{-1},x_0,y_0)$, and $\vec{C}_3$ contains the $2$-arc $(x_{-1},x_0,y_{9d-2})$. Observe that $\rho$ maps the arc $(x_{-1},x_0)$ to the arc $(x_0,x_1)$, and so the fact that $H_1$ is $1$-regular implies that $\rho$ is a shunt of $\vec{C}_1$ in $H_1$. Similarly, $\rho\sigma$ is a shunt of $\vec{C}_2$ in $H_1$ and $\rho\sigma\mu$ is a shunt of $\vec{C}_3$ in $H_1$. Note that by Lemma~\ref{le:cons} any pair of orbits $\mathcal{O}_i$, $i \in \{1,2,3\}$, determines a map $\m$ with underlying graph $\G$, such that $H_1 \leq \Aut(\m)$. We thus only have to check which pairs of orbits are such that $\Aut(\m) = H_1$ and not $\Aut(\m) = \Aut(\G)$ (which occurs if and only if $\tau \in \Aut(\m)$). Since $\tau\rho\tau = \rho\sigma$ and $\tau \rho\mu\sigma\tau = \rho\mu\sigma$ it follows that $\tau$ interchanges the $H_1$-orbits $\mathcal{O}_1$ and $\mathcal{O}_2$ and fixes the orbit $\mathcal{O}_3$ (in fact, it fixes the cycle $\vec{C}_3$). Thus $\tau \in \Aut(\m)$ if and only if the boundaries of faces of $\m$ are all the members of the orbits $\mathcal{O}_1$ and $\mathcal{O}_2$. Consequently, $\m$ is of class~$\2$ if and only if the boundaries of its faces are all the members of $\mathcal{O}_3$ and one of $\mathcal{O}_1, \mathcal{O}_2$. Since $\tau$ fixes $\mathcal{O}_3$ and interchanges $\mathcal{O}_1$ and $\mathcal{O}_2$, this proves that $H_1$ gives rise to exactly one map of class $\2$, up to isomorphism.

In a similar way one can prove that there is exactly one map of class $\2$ with underlying graph $\G$ and automorphism group $H_2$.

It is clear that the two maps of class $\2$, corresponding to $H_1$ and $H_2$, respectively, are not isomorphic since a corresponding isomorphism would have to be an automorphism of $\G$, and so $\G$ is the underlying graph of exactly two maps of class $\2$.
 \hfill $\Box$
\end{proof}



\begin{thebibliography}{9999}



\bibitem{BobMikPot11} M.~Boben, \v S.~Miklavi\v c, P.~Poto\v cnik,
		Rotary polygons in configurations, 
		{\em Electron. J. Combin.} {\bf 18} (2011), P119.
		
\bibitem{STG} G.~Cunningham, M.~Del R\'io-Francos, I.~Hubard, M.~Toledo,
		Symmetry Type Graphs of Polytopes and Maniplexes,
		{\em Annals of Combinatorics} {\bf 19} (2015), 243--268.
		
\bibitem{ruiphd} R. Duarte, Ph.D. Thesis, University of Aveiro, June 2007.

\bibitem{DobKovMik15} E.~Dobson, I.~Kov\'acs, \v S.~Miklavi\v c,
		The automorphism groups of non-edge-transitive rose window graphs,
		{\em Ars Math. Contemp.} {\bf 9} (2015), 63--75.
		
\bibitem{polymani} J.~Garza-Vargas, I.~Hubard,
		Polytopality of maniplexes,
		{\em Preprint arXiv:1604.01164} 
		
\bibitem{edgetrans} J.E.~Graver, M.E.~Watkins, 
		Locally finite, planar, edge-transitive graphs, 
		{\em Mem. Amer. Math. Soc.} {\bf 126} (601) (1997).
		
\bibitem{2-orbit} I.~Hubard,
		Two-orbit polyhedra from groups.
		{\em European J. Combin. } {\bf 1}(3), (2010), 943--960.
		
\bibitem{medial} I.~Hubard, M.~Del R\'io-Francos, A.~Orbanic, T.~Pisanski,
		Medial symmetry type graphs. 
		{\em Electron. J. Combin.} {\bf 20(3)}, (2013), P29.
		
\bibitem{KovKutMar10} I.~Kov\'acs, K.~Kutnar, D.~Maru\v si\v c,
		Classification of edge-transitive rose window graphs,
		{\em J. Graph Theory} {\bf 65} (2010), 216--231.

\bibitem{KovKutRuf10} I.~Kov\'{a}cs, K.~Kutnar, J.~Ruff, 
		Rose window graphs underlying rotary maps, 
		{\em Discrete Math.} {\bf 310} (2010), 1802--1811.
		
\bibitem{ARP} P.~McMullen, E.~Schulte, 
		{\em Abstract Regular Polytopes.}
		 Encyclopedia Math. Appl. 92. Cambridge University Press, Cambridge (2002)
 
\bibitem{Mik13} \v S.~Miklavi\v c,
		A note on a conjecture on consistent cycles,
		{\em Ars Math. Contemp.} {\bf 6} (2013), 389--392.
		
\bibitem{MikPotWil07} \v S.~Miklavi\v c, P.~Poto\v cnik, S.~Wilson,
		Consistent cycles in graphs and digraphs, 
		{\em Graphs Combin.} {\bf 23} (2007), 205--216.
		
\bibitem{hered} M.~Mixer, E.~Schulte, A.~Weiss,
		Hereditary polytopes
		{\em Rigidity and Symmetry} Volume 70 of the series Field Institute Communications. (2014), 279--302.
		
\bibitem{PraXu89} C.~E.~Praeger, M.~Y.~Xu. A characterization of a class of symmetric graphs of twice prime valency, 
		{\em European J. Combin.} {\bf 10} (1989), 91--102.
		
\bibitem{Wil08} S.~Wilson,
		Rose Window Graphs,
		{\em Ars Math. Contemp.} {\bf 1} (2008), 7--19.
		
\end{thebibliography}
\end{document}